\documentclass[10pt]{amsart}

\usepackage{xcolor}
\definecolor{winered}{rgb}{0.8,0,0}
\definecolor{deepblue}{rgb}{0,0,0.8}
\usepackage[colorlinks,linkcolor=black,citecolor=black,urlcolor=black]{hyperref}
\usepackage{todonotes}
\usepackage{amsmath}
\usepackage{amsthm}
\usepackage{amssymb}
\usepackage{mathrsfs}
\usepackage{tikz-cd}
\newtheorem{thm}{Theorem}[section]
\newtheorem{prop}[thm]{Proposition}
\newtheorem{cor}[thm]{Corollary}

\newtheorem{lem}[thm]{Lemma}
\theoremstyle{definition}
\newtheorem{df}[thm]{Definition}
\newtheorem{rmk}[thm]{Remark}
\newtheorem{exm}[thm]{Example}

\newtheorem{const}[thm]{Construction}

\theoremstyle{remark}
\numberwithin{equation}{section}

\newcommand{\bP}{\mathbf{P}}

\newcommand{\A}{\mathbb{A}}

\newcommand{\E}{\mathbb{E}}

\newcommand{\N}{\mathbb{N}}
\renewcommand{\P}{\mathbb{P}}

\newcommand{\Sphere}{\mathbb{S}}

\newcommand{\Z}{\mathbb{Z}}

\newcommand{\cC}{\mathcal{C}}
\newcommand{\cD}{\mathcal{D}}
\newcommand{\cE}{\mathcal{E}}
\newcommand{\cF}{\mathcal{F}}

\newcommand{\cM}{\mathcal{M}}

\newcommand{\cO}{\mathcal{O}}

\newcommand{\rE}{\mathrm{E}}

\newcommand{\et}{{\acute{e}t}}

\DeclareMathOperator{\Spec}{Spec}

\newcommand{\colim}{\mathop{\mathrm{colim}}}

\newcommand{\id}{\mathrm{id}}
\newcommand{\ul}{\underline}
\newcommand{\ol}{\overline}
\newcommand{\cofib}{\mathrm{cofib}}
\newcommand{\fib}{\mathrm{fib}}

\newcommand{\Sh}{\mathrm{Sh}}
\newcommand{\PSh}{\mathrm{PSh}}

\newcommand{\Sp}{\mathrm{Sp}}

\newcommand{\Sm}{\mathrm{Sm}}

\newcommand{\CAlg}{\mathrm{CAlg}}

\newcommand{\NAlg}{\mathrm{NAlg}}

\newcommand{\lSch}{\mathrm{lSch}}

\newcommand{\SmlSm}{\mathrm{SmlSm}}
\newcommand{\Fun}{\mathrm{Fun}}

\newcommand{\Span}{\mathrm{Span}}
\newcommand{\Fin}{\mathrm{Fin}}
\newcommand{\FinGpd}{\mathrm{FinGpd}}

\newcommand{\tfib}{\mathrm{tfib}}
\newcommand{\tcofib}{\mathrm{tcofib}}

\newcommand{\LEq}{\mathrm{LEq}}

\newcommand{\pt}{\mathrm{pt}}

\newcommand{\EM}{\mathrm{H}}

\newcommand{\logSH}{\mathrm{logSH}}

\newcommand{\SH}{\mathrm{SH}}

\newcommand{\Normal}{\mathrm{N}}
\newcommand{\Blow}{\mathrm{Bl}}
\newcommand{\Th}{\mathrm{Th}}

\newcommand{\Deform}{\mathrm{D}}

\newcommand{\Bcy}{\mathrm{B}^\mathrm{cy}}
\newcommand{\Brep}{\mathrm{B}^\mathrm{rep}}
\newcommand{\Bdi}{\mathrm{B}^{\mathrm{di}}}

\newcommand{\Bdrep}{\mathrm{B}^\mathrm{drep}}
\newcommand{\Ncy}{\mathrm{N}^\mathrm{cy}}
\newcommand{\Ndrep}{\mathrm{N}^\mathrm{drep}}
\newcommand{\Ndi}{\mathrm{N}^\mathrm{di}}

\newcommand{\Nsigma}{\mathrm{N}^\sigma}
\newcommand{\CycSp}{\mathrm{CycSp}}
\newcommand{\RCycSp}{\mathbb{R}\mathrm{CycSp}}
\newcommand{\THH}{\mathrm{THH}}
\newcommand{\THR}{\mathrm{THR}}
\newcommand{\TC}{\mathrm{TC}}
\newcommand{\TCR}{\mathrm{TCR}}
\newcommand{\TPR}{\mathrm{TPR}}
\newcommand{\sd}{\mathrm{sd}}

\newcommand{\bTHH}{\mathbf{THH}}
\newcommand{\bTCR}{\mathbf{TCR}}
\newcommand{\bTC}{\mathbf{TC}}
\newcommand{\bTHR}{\mathbf{THR}}
\newcommand{\prelogSH}{\mathrm{prelogSH}}
\newcommand{\Kth}{\mathrm{K}}

\newcommand{\gp}{\mathrm{gp}}

\definecolor{indigo}{rgb}{0.3, 0, 0.5}

\newcommand{\Ztwo}{{\mathbb{Z}/2}}
\newcommand{\Cat}{\mathrm{Cat}}

\begin{document}
\author{Doosung Park}
\address{Bergische Universit{\"a}t Wuppertal,
Fakult{\"a}t Mathematik und Naturwissenschaften
\\
Gau{\ss}strasse 20, 42119 Wuppertal, Germany}
\email{dpark@uni-wuppertal.de}

\title{Motivic real topological Hochschild spectrum}
\begin{abstract}
We define real topological Hochschild homology of separated log schemes with involutions.
We show that real topological Hochschild homology is $(\P^n,\P^{n-1})$-invariant,
which leads to the definition of the motivic real topological Hochschild spectrum living in a certain $\Ztwo$-equivariant logarithmic motivic category.
We explore properties of real topological Hochschild homology that can be deduced from the logarithmic motivic homotopy theory.
We also define the motivic real topological cyclic spectrum.
\end{abstract}
\maketitle

\section{Introduction}

Topological Hochschild homology $\THH$ and its cousin $\TC$ have a deep connection with algebraic $K$-theory via the cyclotomic trace of B\"okstedt-Hsiang-Madsen \cite{BHM93}.
The Dundas-Goodwillie-McCarthy theorem \cite{DGM13} provided a computational tool for algebraic $K$-theory.
It was also discovered that $\THH$ contains arithmetic data.
Bhatt, Morrow, and Scholze \cite{BMS19} studied filtrations on the $S^1$-homotopy fixed point spectrum $\TC^-$ of $\THH$,
whose graded pieces are closely related to the prismatic cohomology of Bhatt-Scholze according to \cite[\S 13]{zbMATH07611906}.

\medskip

Hesselholt and Madsen \cite{HM} defined real algebraic $K$-theory,
which refines both algebraic and hermitian $K$-theories in a uniform manner via equivariant homotopy theory.
They also defined real topological Hochschild homology $\THR$ of rings.
For further developments on $\THR$,
we refer to \cite{Hog}, \cite{DMPR21}, and \cite{QS2}.
A forthcoming work of Harpaz, Nikolaus, and Shah will show a real refinement of the Dundas-Goodwillie-McCarthy theorem,
which would justify the real trace method as a computational tool for real or hermitian $K$-theory.

\medskip

One feature of $\THH$ different from algebraic $K$-theory is that $\THH$ is not $\A^1$-invariant even for regular schemes,
i.e.,
the induced map $\THH(X)\to \THH(X\times \A^1)$ is not an equivalence of spectra for every nonempty scheme $X$.
The definition of $\THH$ of schemes is due to Geisser and Hesselholt \cite{GH}.
Hence the motivic methods in $\A^1$-homotopy theory initiated by Morel and Voevodsky \cite{MV} are not directly applicable to $\THH$.

\medskip

Non $\A^1$-invariance of $\THH$ is related to the fact that the sequence of spectra
\[
\THH(Z)\xrightarrow{i_*} \THH(X)
\xrightarrow{j^*} \THH(X-Z)
\]
is not a fiber sequence when $i\colon Z\to X$ is a closed immersion of regular schemes and $j\colon X-Z\to X$ is its open complement.
On the other hand,
the localization sequence in algebraic $K$-theory can be read as the fiber sequence of spectra
\[
\Kth(Z)\xrightarrow{i_*} \Kth(X)\xrightarrow{j^*} \Kth(X-Z).
\]
Hesselholt-Madsen \cite{MR1998478} and Rognes-Sagave-Schlichtkrull \cite{MR3412362} studied localization sequences for $\THH$ in the logarithmic setting.

\medskip

The author's joint work with Binda and {\O}stv{\ae}r \cite{logSH} introduced the $\infty$-category of logarithmic motivic spectra $\logSH(S)$ for fs log schemes $S$,
which aims to incorporate various non $\A^1$-invariant cohomology theories into the motivic framework using logarithmic geometry.
We refer to Ogus's book \cite{Ogu} for logarithmic geometry.
For a closed immersion of schemes $Z\to X$,
let $(X,Z)$ denote the log scheme with the underlying scheme $X$ and with the compactifying log structure associated with the open immersion $X-Z\to X$.
In the construction of $\logSH(S)$ when $S$ is a scheme (with the trivial log structure),
the interval $\A^1$ is replaced with the set of log schemes $(\P^n,\P^{n-1})$ for all integers $n\geq 1$,
and the Nisnevich topology is replaced with the strict Nisnevich topology in \cite{logDM}.
To construct a motivic spectrum in $\logSH(S)$ that represents an existing cohomology theory of schemes,
one can take the following two steps:
\begin{itemize}
\item[Step 1.] Extend the cohomology theory to fs log schemes.
\item[Step 2.] Show that the extension is $(\P^n,\P^{n-1})$-invariant for all integers $n\geq 1$ and satisfies strict Nisnevich descent.
\end{itemize}
Rognes \cite{Rog09} defines $\THH$ of log rings.
According to \cite[\S 8]{logSH},
it is possible to define $\THH$ and $\TC$ of log schemes based on Rognes' definition and to construct a $\P^1$-spectrum $\bTHH$ and $\bTC$ (written as $\mathbf{logTHH}$ and $\mathbf{logTC}$ in loc.\ cit.) representing $\THH$ and $\TC$.
The author's joint work with Binda, Lundemo, and {\O}stv{\ae}r \cite[\S 8]{BLPO} constructed a $\P^1$-spectrum representing Hochschild homology.

\medskip

In this article,
we employ this strategy for $\THR$.
We define $\THR$ of log rings with involutions,
where an involution means an automorphism $\sigma$ such that $\sigma\circ \sigma=\id$.
With the aid of the author's joint work with Hornbostel \cite{THR},
we define $\THR$ of separated log schemes with involutions.
Then we prove that $\THR$ is $(\P^n,\P^{n-1})$-invariant for all integers $n\geq 1$.
This allows us to apply motivic methods in logarithmic motivic homotopy theory to $\THR$.
We also consider real topological cyclic homology,
which is a variant of $\THR$.
One way to obtain $\TCR$ is to impose the real cyclomotic structure on $\THR$, which is due to Quigley and Shah \cite{QS2}.
The following two results are examples of this.

\medskip

\begin{thm}
[Special case of Theorem \ref{mot.4}]
\label{intro.1}
Let $Z\to X$ be a closed immersion of smooth schemes over a noetherian separated scheme $S$.
Then there exist fiber sequences of $\Ztwo$-spectra
\begin{gather*}
\THR(\Th(\Normal_Z X))
\to
\THR(X)
\to
\THR(\Blow_Z X,E),
\\
\TCR(\Th(\Normal_Z X))
\to
\TCR(X)
\to
\TCR(\Blow_Z X,E),
\end{gather*}
where $E$ is the exceptional divisor.
\end{thm}

We refer to \eqref{mot.15.1} for the notation $\THR(\Th(\Normal_Z X))$.
Proposition \ref{mot.15} implies that
$\THR(\Th(\Normal_Z X))$ can be written in terms of $\THR$ of schemes.
This sequence can be considered as the \emph{localization sequence for $\THR$}.
The map $\THR(\Th(\Normal_Z X)) \to \THR(X)$ is called the \emph{Gysin map}.

\begin{thm}
[Theorem \ref{mot.5}]
\label{intro.2}
Let $Z\to X$ be a closed immersion of smooth schemes over a finite dimensional noetherian separated scheme $S$.
Then the induced squares of $\Ztwo$-spectra
\[
\begin{tikzcd}
\THR(X)\ar[d]\ar[r]&
\THR(Z)\ar[d]
\\
\THR(Z\times_X \Blow_Z X)\ar[r]&
\THR(\Blow_Z X),
\end{tikzcd}
\begin{tikzcd}
\TCR(X)\ar[d]\ar[r]&
\TCR(Z)\ar[d]
\\
\TCR(Z\times_X \Blow_Z X)\ar[r]&
\TCR(\Blow_Z X)
\end{tikzcd}
\]
are cartesian,
where $\Blow_Z X$ denotes the blow-up of $X$ along $Z$.
\end{thm}

The statement in this theorem involves no log schemes,
but the proof uses log schemes.

\medskip

Hu, Kriz, and Ormsby \cite{HuKO} considered $\P^1\wedge \P^\sigma$-spectra instead of $\P^1$-spectra to define a motivic $\Ztwo$-spectrum,
where $\P^\sigma$ is the scheme $\P^1$ with the involution given by $[x:y]\mapsto [y:x]$.
In Definition \ref{mot.8},
we analogously introduce the $\infty$-category of prelogarithmic motivic $\Ztwo$-spectra $\prelogSH^{\Ztwo}(S)$ for any scheme $S$,
which is helpful for constructing logarithmic motivic spectra using the fixed point functor
\[
(-)^{\Ztwo}
\colon
\prelogSH^{\Ztwo}(S)
\to
\logSH(S)
\]
in Definition \ref{mot.10}.
We also construct the ``forgetful'' functor 
\[
i^*
\colon
\prelogSH^{\Ztwo}(S)
\to
\logSH(S)
\]
in Definition \ref{mot.13}.
The reason why we add ``pre'' in the notation $\prelogSH^{\Ztwo}(S)$ is that a further localization is desired to obtain a better behaved $\infty$-category of logarithmic motivic $\Ztwo$-spectra,
see Remark \ref{mot.9} for an explanation.

\medskip

Together with the $\P^1\wedge \P^\sigma$-periodicity of $\THR$ in \cite[Proposition 5.1.5]{THR},
we can define the periodic  $\P^1\wedge \P^\sigma$-spectra
\begin{gather*}
\bTHR
:=
(\THR^{\Z/2},\THR^{\Z/2},\cdots)
\in
\prelogSH^{\Ztwo}(S),
\\
\bTCR
:=
(\TCR^{\Z/2},\TCR^{\Z/2},\cdots)
\in
\prelogSH^{\Ztwo}(S)
\end{gather*}
for every finite dimensional noetherian separated scheme $S$.
The construction of $\bTHR$ resembles the construction of the motivic real $K$-theory $\Ztwo$-spectrum \cite[\S 4.3]{HuKO}.
We obtain new $\P^1$-spectra $\bTHR^{\Ztwo}$ and $\bTCR^{\Ztwo}$,
while we show $i^*\bTHR\simeq \bTHH$ and $i^*\bTCR\simeq \bTC$ in Proposition \ref{mot.14}.

\subsection*{Organization of the article}

To define the $\P^1\wedge \P^\sigma$-spectrum $\bTHR$,
we need the following ingredients:

\begin{itemize}
\item[(1)]
The definition of $\THR$ of log rings with involutions.
\item[(2)]
Isovariant \'etale descent property for $\THR$.
\item[(3)]
Invariance of $\THR$ under passing to the associated log structure.
\item[(4)]
$\P^1\wedge \P^\sigma$-periodicity of $\THR$.
\item[(5)]
$(\P^n,\P^{n-1})$-invariance of $\THR$.
\end{itemize}

Due to \cite{THR},
we have (2) and (4).
In \S \ref{thrring},
we deal with (1).
This requires dihedral replete bar constructions in \S \ref{dih},
which is an equivariant analogue of Rognes' replete bar constructions \cite[Definition 3.16]{Rog09}.
In \S \ref{thrscheme},
we show (3) for separated log schemes.
The strategy is to work in a sufficiently local situation.
In \S \ref{TCR},
we review the notion of real cyclotomic spectra following Quigley and Shah \cite{QS2}.
Then we generalize many results in the previous sections to real cyclotomic spectra so that we can define $\TCR$ of finite dimensional noetherian separated log schemes.
In \S \ref{mot},
we show (5) by providing an explicit computation of $\THR(\P^n,\P^{n-1})$ using cubes in $\infty$-categories.
Then we discuss properties of $\THR$ that can be deduced from \cite{logSH},
and we construct the $\P^1\wedge \P^\sigma$-spectra $\bTHR$ and $\bTC$.

\subsection*{Acknowledgement}
This research was conducted in the framework of the DFG-funded research training group GRK 2240: \emph{Algebro-
Geometric Methods in Algebra, Arithmetic and Topology}.
We thank Jens Hornbostel for many helpful conversations on this topic.
We are grateful to the referee for detailed and precise comments leading to improvements of this text.

\section{Dihedral replete bar constructions}
\label{dih}

Rognes \cite[Definition 3.16]{Rog09} defined the replete bar construction of a commutative monoid,
which is a key input to define topological Hochschild homology of log rings.
In this section,
we discuss a dihedral refinement of this construction.

We refer to \cite[Example 4.1.3]{THR} for a review of real simplicial sets and dihedral sets,
which are defined using crossed simplicial groups of Fiedorowicz-Loday \cite{FL}.
A \emph{real simplicial set} $X$ is a simplicial set equipped with involutions $w_q\colon X_q \to X_q$ (i.e., an automorphism such that $w_q\circ w_q=\id$) for all simplicial degrees $q$ satisfying
the relations
\[
d_iw_{q}=w_{q-1}d_{q-i},
\text{ }
s_iw_q=w_{q+1}s_{q-i}
\]
for $0\leq i\leq q$.
Note that a real simplicial set is different from a simplicial set with involution, i.e., a simplicial object in the category of sets with involution.
A \emph{dihedral simplicial set} $X$ is a real simplicial set equipped with automorphisms $t_q\colon X_q\to X_q$ for all simplicial degrees $q$ satisfying certain relations.

Let $\oplus$ denote the coproduct in the category of commutative monoids.
For commutative monoids $P$ and $Q$,
the coproduct $P\oplus Q$ is naturally isomorphic to the product $P\times Q$.

\begin{df}
\label{drep.5}
Let $\Delta_\sigma^1$ be the real simplicial set whose underlying simplicial set is $\Delta^1$ and whose involution $w\colon (\Delta^1)_q\to (\Delta^1)_q$ in simplicial degree $q$ sends the $q$-simplex $a_0 \cdots a_q$ to $(1-a_0)\cdots (1-a_q)$.
For a commutative monoid $P$ with involution $w$,
let $P\otimes \Delta_\sigma^1$ be the real simplicial set whose underlying simplicial set is the tensor product $P\otimes \Delta^1$ defined by the formula
\[
(P\otimes \Delta^1)_q
:=
\bigoplus_{i\in (\Delta^1)_q}P
\]
and whose involution is given by $(x_0,\ldots,x_{q+1})\to (w(x_{q+1}),\ldots,w(x_0))$ in simplicial degree $q$,
where the indices $0,1,\ldots,q+1$ correspond to the $q$-simplicies $0\cdots 00,0\cdots 01,\ldots,1\cdots 11$ in $\Delta^1$.
\end{df}

\begin{prop}
\label{drep.1}
Let $P$ be a commutative monoid with involution.
Then the map of real simplicial sets
\[
P\otimes \Delta_\sigma^1
\to
P
\]
given by $(x_0,\ldots,x_q)\mapsto x_0+\cdots+x_q$ in simplicial degree $q$ is a $\Ztwo$-weak equivalence.
\end{prop}
\begin{proof}
Let $\sd_\sigma$ be the Segal subdivision functor \cite{Seg73}.
We have an isomorphism of simplicial sets with involutions
\[
\sd_\sigma(\Delta_\sigma^1)\simeq \Delta^1\amalg_{\{1\}} \Delta^1
\]
where the involution on the right-hand side switches the factors.
This induces an isomorphism of simplicial sets with involutions
\[
\sd_\sigma(P\otimes \Delta_\sigma^1)
\simeq
(P\otimes \Delta^1) \oplus_P (P\otimes \Delta^1),
\]
where the involution on the right-hand side is obtained by the formula $(x,y)\mapsto (w(y),w(x))$ with $w$ in Definition \ref{drep.5}.
Since $\Delta^1$ is contractible,
the map $P\to P\otimes \Delta^1$ induced by $\{1\}\to \Delta^1$ is a homotopy equivalence.
This yields a $\Ztwo$-homotopy equivalence
\[
P
\xrightarrow{\simeq}
(P\otimes \Delta^1) \oplus_P (P\otimes \Delta^1),
\]
which implies the claim.
\end{proof}

We review the $\infty$-categorical formulation of equivariant homotopy theory due to Bachmann-Hoyois \cite[\S 9]{BH21}.
See \cite[\S A.1]{THR} for a more detailed review.
Let $\FinGpd$ denote the $2$-category of finite groupoids.
Bachmann and Hoyois constructed the functor
\[
\SH
\colon
\Span(\FinGpd)
\to
\CAlg(\Cat_\infty),
\;
(X\xleftarrow{f} Y\xrightarrow{p} Z)
\mapsto
p_\otimes f^*,
\]
see \cite[\S C]{BH21} for the notation $\Span$.
For a morphism $f$ in $\FinGpd$,
$f^*$ admits a right adjoint $f_*$.
If $f$ is a finite covering,
then $f^*$ admits a left adjoint $f_\sharp$.

For a finite groupoid $X$,
let $\NAlg(\SH(X))$ be the $\infty$-category of normed $X$-spectra \cite[Definition 9.14]{BH21},
and let $\CAlg(\SH(X))$ be the $\infty$-category of $\E_\infty$-rings in $\SH(X)$.
We have the forgetful functor $\NAlg(\SH(X))\to \CAlg(\SH(X))$.
We use the notation $\wedge$ for the coproduct in $\NAlg(\SH(X))$ and $\CAlg(\SH(X))$.
For a morphism of finite groupoids $f\colon X\to S$,
we have the induced adjoint functors
\[
\begin{tikzcd}
\CAlg(\SH(S))
\ar[r,shift left=0.75ex,"f^*"]
\ar[r,shift right=0.75ex,"f_*"',leftarrow]
&
\CAlg(\SH(X)).
\end{tikzcd}
\]
If $f\colon X\to S$ is a finite covering of finite groupoids,
then we have the induced sequence of adjoint functors
\[
\begin{tikzcd}[column sep=large]
\NAlg(\SH(S))
\ar[r,shift left=1.5ex,leftarrow,"f_\sharp"]
\ar[r,"f^*" description]
\ar[r,"f_*"',shift right=1.5ex,leftarrow]
&
\NAlg(\SH(X)).
\end{tikzcd}
\]
On the other hand,
if $f$ has connected fibers,
then the induced functor
\[
f_\otimes \colon \CAlg(\SH(X))\to \CAlg(\SH(S))
\]
preserves colimits.

For a finite group $G$,
the \emph{$\infty$-category of $G$-spectra} is
\[
\Sp^G
:=
\SH(B G).
\]
We also set $\NAlg^G:=\NAlg(\SH(B G))$ and $\CAlg^G:=\CAlg(\SH(B G))$.
We omit the superscripts $G$ in this notation if $G$ is trivial.
There is an equivalence of $\infty$-categories $\CAlg\simeq \NAlg$.

Consider the obvious functors $\pt \xrightarrow{i} B (\Ztwo) \xrightarrow{p} \pt$.
Observe that $i$ is a finite covering and $p$ has connected fibers.
We often use the alternative notation
\[
(-)^{\Ztwo}
:=
p_*,
\;
N^{\Ztwo}
:=
i_\otimes,
\;
\Phi^{\Ztwo}
:=
p_\otimes.
\]
The functor $N^{\Ztwo}$ is the \emph{norm functor} of Hill-Hopkins-Ravenel \cite{HHR},
and the functor $\Phi^{\Ztwo}$ is the \emph{geometric fixed point functor}.
We have the induced functors
\begin{gather*}
i^*\colon \Sp^{\Ztwo} \to \Sp,
\;
N^{\Ztwo},i_\sharp,i_* \colon \Sp\to \Sp^{\Ztwo},
\\
i^*\colon \NAlg^{\Ztwo} \to \CAlg,
\;
N^{\Ztwo},i_\sharp,i_* \colon \CAlg\to \NAlg^{\Ztwo},
\\
p^*\colon \Sp\to \Sp^{\Ztwo},
\;
\Phi^{\Ztwo},(-)^{\Ztwo} \colon \Sp^{\Ztwo}\to \Sp
\\
p^*\colon \CAlg\to \CAlg^{\Ztwo},
\;
\Phi^{\Ztwo},(-)^{\Ztwo} \colon \CAlg^{\Ztwo}\to \CAlg,
\;
\end{gather*}
We refer to \cite[Proposition A.2.7]{THR} for fundamental relations among these functors.
We note that the pair of functors $(i^*,\Phi^{\Ztwo})$ is conservative.

Let $f\colon X\to S$ be a morphism of finite groupoids.
For every map $R\to A$ in $\NAlg(\SH(S))$ and map $f^*R\to B$ in $\NAlg(\SH(X))$,
we have the natural map
\[
A\wedge_R f_*B
\to
f_*(f^*A\wedge_{f^*R}B)
\]
given by the composite
\[
A\wedge_R f_*B
\xrightarrow{ad}
f_*f^*(A\wedge_R f_*B)
\xrightarrow{\simeq}
f_*(f^*A \wedge_{f^*R} f^*f_*B)
\xrightarrow{ad'}
f_*(f^*A \wedge_{f^*R} B),
\]
where $ad$ (resp.\ $ad'$) denotes the map obtained by the unit (resp.\ counit).

\begin{lem}
\label{thrlog.7}
Let $r\colon \Ztwo\to \pt$ be the obvious morphism of finite groupoids.
Then the natural map
\[
A\wedge_R r_*B
\to
r_*(r^*A\wedge_{r^*R} B)
\]
is an equivalence for every map $R\to A$ in $\CAlg$ and map $r^*R\to B$ in $\NAlg(\SH(\Ztwo))$.
\end{lem}
\begin{proof}
As observed in \cite[Example 9.15]{BH21},
we have an equivalence of $\infty$-categories $\NAlg(\SH(\Z/2))\simeq \CAlg(\SH(\Z/2))$.
We also have an equivalence of $\infty$-categories $\SH(\Z/2)\simeq \SH(\pt) \times \SH(\pt)$ by \cite[Lemma 9.6]{BH21} and hence $\CAlg(\SH(\Z/2))\simeq \CAlg\times \CAlg$.
The functors $i_0^*,i_1^*\colon \NAlg(\SH(\Ztwo))\to \CAlg$ induced by the two inclusions $i_0,i_1\colon \pt\to \Z/2$ can be identified with the two projections $\CAlg\times \CAlg\to \CAlg$.
Using $ri_0=ri_1$,
we see that the functor $r^*\colon \CAlg\to \NAlg(\SH(\Ztwo))$ can be identified with the diagonal functor $\CAlg \to \CAlg\times \CAlg$, and the functor $r_*\colon \NAlg(\SH(\Ztwo))\to \CAlg$ can be identified with the direct sum functor $\oplus \colon \CAlg\times \CAlg\to \CAlg$.
Use these explicit descriptions to show the claim.
\end{proof}

\begin{lem}
\label{thrlog.6}
For $R,A\in \NAlg^{\Ztwo}$ and $B\in \CAlg$,
the natural map
\begin{equation}
\label{thrlog.6.1}
A\wedge_R i_*B
\to
i_*(i^*A\wedge_{i^*R}B)
\end{equation}
is an equivalence.
\end{lem}
\begin{proof}
It suffices to show that \eqref{thrlog.6.1} becomes equivalences after applying $i^*$ and $\Phi^{\Ztwo}$.
With the aid of the forgetful functor $\NAlg^{\Ztwo}\to \CAlg^{\Ztwo}$, we obtain an equivalence $\Phi^{\Ztwo}(A\wedge_R i_*B)\simeq \Phi^{\Ztwo}A \wedge_{\Phi^{\Ztwo}R}\Phi^{\Ztwo}i_*B$.
By \cite[Proposition A.2.7(3),(5)]{THR},
we have $\Phi^{\Ztwo}i_*\simeq 0$.
Hence both sides of \eqref{thrlog.6.1} vanish after applying $\Phi^{\Ztwo}$.

Apply \cite[Proposition A.1.9]{THR} to the cartesian square
\[
\begin{tikzcd}
\Ztwo\ar[d,"r"']\ar[r,"r"]&
\pt\ar[d,"i"]&
\\
\pt\ar[r,"i"]&
B (\Ztwo)
\end{tikzcd}
\]
to obtain a natural equivalence $r_*r^*\simeq i^*i_*$.
We have natural equivalences
\[
r_*r^*((-)\wedge_{(-)} (-))
\simeq
r_*(r^*(-)\wedge_{r^*(-)}r^*(-))
\simeq 
(-)\wedge_{(-)}r_*r^*(-),
\]
where the second one is due to Lemma \ref{thrlog.7}.
We have the induced commutative diagram
\[
\begin{tikzcd}[column sep=tiny]
i^*(A\wedge_R i_*B)\ar[r,"ad"]\ar[d,"\simeq"']&
i^*i_*i^*(A\wedge_R i_*B)\ar[r,"\simeq"]&
i^*i_*(i^*A\wedge_{i^*R} i^*i_*B)\ar[r,"ad'"]\ar[d,"\simeq"]&
i^*i_*(i^*A\wedge_{i^*R} B)\ar[d,"\simeq"]
\\
i^*A\wedge_{i^*R}i^*i_*B\ar[rr,"ad"]\ar[rrd,"ad"']&
&
r_*r^*(i^*A\wedge_{i^*R}i^*i_*B)\ar[r,"ad'"]\ar[d,"\simeq"]&
r_*r^*(i^*A\wedge_{i^*R}B)\ar[d,"\simeq"]
\\
&
&
i^*A\wedge_{i^*R}r_*r^*i^*i_*B\ar[r,"ad'"]&
i^*A\wedge_{i^*R}r_*r^*B.
\end{tikzcd}
\]
Hence to show that \eqref{thrlog.6.1} becomes an equivalence after applying $i^*$, it suffices to show that the composite of the upper vertical maps in the commutative diagram
\[
\begin{tikzcd}
i^*i_*B\ar[r,"ad"]\ar[d,"\simeq"']&
r_*r^*i^*i_*B\ar[r,"ad'"]\ar[d,"\simeq"]&
r_*r^*B\ar[d,"\simeq"]
\\
r_*r^*B\ar[r,"ad"]&
r_*r^*r_*r^*B\ar[r,"ad'"]&
r_*r^*B
\end{tikzcd}
\]
is an equivalence.
This is true since the composite of the lower horizontal maps is an equivalence by the counit-unit identity.
\end{proof}

For a simplicial set (resp.\ real simplicial set) $X$,
we use the notation $\Sphere[X]:=\Sigma^\infty X_+$,
which is a spectrum (resp.\ $\Ztwo$-spectrum).
If $P$ is a commutative monoid,
then $\Sphere[P]$ is an object of $\CAlg$.
If $P$ is a commutative monoid with involution,
then $\Sphere[P]$ is a commutative monoid in the model category of orthogonal $\Ztwo$-spectrum in the sense of \cite[Definition A.2.2]{THR}.
Due to \cite[Remark A.2.6]{THR},
we can regard $\Sphere[P]$ as an object of $\NAlg^{\Ztwo}$.

For a commutative monoid $P$ with involution,
let $\Ndi P$ be the \emph{dihedral nerve of $P$} in \cite[Definition 4.2.2]{THR}, and let $\Bdi P$ be its dihedral geometric realization.
In simplicial degree $q$,
we have $(\Ndi P)_q:=P^{\times (q+1)}$.
We obtain the cyclic nerve $\Ncy i^* P$ \cite[\S 2.3]{zbMATH03671343} of $i^*P$ if we forget the involution structure on $\Ndi P$.
There is a map of dihedral sets
\begin{equation}
\label{drep.2.1}
\Ndi P
\to
P
\end{equation}
sending $(x_0,\ldots,x_q)$ to $x_0+\cdots+x_q$ in simplicial degree $q$.

\begin{df}
\label{thrlog.9}
Let $\cC$ be an ordinary category.
An \emph{object of $\cC$ with involution} is an object of $\cC_{\Ztwo}:=\Fun(B (\Ztwo),\cC)$.
Let
\[
i^*\colon \Fun(B (\Ztwo),\cC)
\to
\cC
\]
be the forgetful functor,
and let $i_\sharp$ (resp.\ $i_*$) be its left adjoint (resp.\ right adjoint) if it exists.
\end{df}

\begin{exm}
If $P$ is a commutative monoid,
then $i_*P$ is the commutative monoid $P\times P$ with the involution $w$ given by $w(x,y):=(w(y),w(x))$.
Furthermore,
there is a natural isomorphism $i_\sharp P\simeq i_*P$.
We have a similar description of $i_*A$ for any commutative ring $A$,
but $i_\sharp A \not\simeq i_*A$.

If $X$ is a scheme,
then $i_\sharp X$ is the scheme $X\amalg X$ with the involution switching the two components.
\end{exm}

\begin{lem}
\label{drep.6}
Let $P\to Q$ be a map of commutative monoids with involutions.
Then there is a natural equivalence of $\Ztwo$-spectra
\[
\Sphere[P\amalg P]\wedge_{\Sphere[P]}\Sphere[Q]
\simeq
\Sphere[(P\amalg P)\oplus_P Q],
\]
where the involution on $P\amalg P$ in the formulation switches the components.
\end{lem}
\begin{proof}
There is a natural isomorphism of sets with involutions
\[
(P\amalg P) \oplus_P Q
\simeq
Q\amalg Q,
\]
where the involution on the right-hand side switches the components.
By the explicit description of the functors $i^*$ and $i_*$ in terms of orthogonal spectra in \cite[Construction A.2.4]{THR},
we have natural equivalences
\[
\Sphere[P\amalg P]
\simeq
i_*i^*\Sphere[P]
\text{ and }
\Sphere[Q\amalg Q]
\simeq
i_*i^*\Sphere[Q].
\]
Lemma \ref{thrlog.6} yields a natural equivalence
\[
i_*i^*\Sphere[P]
\wedge_{\Sphere[P]}
\Sphere[Q]
\simeq
i_*(i^*\Sphere[P]\wedge_{i^*\Sphere[P]} i^*\Sphere[Q]).
\]
Combine what we have discussed above to obtain the desired equivalence.
\end{proof}

\begin{prop}
\label{drep.2}
Let $P$ be a commutative monoid with involution.
Then there is a natural equivalence of $\Ztwo$-spectra
\[
\Sphere[\Bdi P]
\simeq
\Sphere[P]\wedge_{\Sphere[i_\sharp i^*P]} \Sphere[P],
\]
where the homomorphism $i_\sharp i^*P\to P$ in the formulation is obtained by the counit of the adjunction pair $(i_\sharp,i^*)$.
\end{prop}
\begin{proof}
There is an isomorphism of real simplicial sets
\begin{equation}
\label{drep.2.2}
(P\otimes \Delta_\sigma^1)\oplus_{i_\sharp i^*P} P
\simeq
\Ndi P,
\end{equation}
where the map $i_\sharp i^*P\to P\otimes \Delta_\sigma^1$ is given by $(x,y)\mapsto (x,0,\ldots,0,y)$ in simplicial degree $q$.
The composite
\[
i_\sharp i^*P \to P\otimes \Delta_\sigma^1 \to P
\]
coincides with the counit homomorphism.
Since $P\otimes \Delta_\sigma^1$ is degreewise the disjoint union of finitely many copies of $i_\sharp i^* P$ and $i_\sharp i^*P \amalg i_\sharp i^*P$ with the switching involution as an $i_\sharp i^*P$-set,
Lemma \ref{drep.6} yields a natural equivalence
\[
\Sphere[P\otimes \Delta_\sigma^1]
\wedge_{\Sphere[i_\sharp i^* P]}
\Sphere[P]
\simeq
\Sphere[(P\otimes \Delta_\sigma^1)\oplus_{i_\sharp i^* P} P].
\]
Use Proposition \ref{drep.1} and \eqref{drep.2.2} to finish the proof.
\end{proof}

For a commutative monoid $P$,
let $P^\gp$ denote its group completion.

\begin{df}
\label{dih.13}
Let $P$ be a commutative monoid with involution.
The \emph{dihedral replete nerve of $P$} is the dihedral set
\begin{equation}
\label{dih.13.1}
\Ndrep P
:=
\Ndi P^\gp \times_{P^\gp} P,
\end{equation}
where the map $\Ndi P^\gp\to P^\gp$ in this formulation is given by \eqref{drep.2.1} for $P^\gp$.
The \emph{dihedral replete bar construction of $P$}, denoted $\Bdrep P$,
is the dihedral geometric realization of $\Ndrep P$.

We obtain the replete bar construction $\Bdi i^*P$ \cite[Definition 3.16]{Rog09} of $i^*P$ if we forget the involution structure on $\Bdrep P$.
\end{df}

\begin{prop}
\label{drep.4}
Let $P$ and $Q$ be commutative monoids with involutions.
Then there is a natural isomorphism of dihedral sets
\[
\Ndrep (P\times Q)
\simeq
\Ndrep P \times \Ndrep Q.
\]
\end{prop}
\begin{proof}
By \cite[Proposition 4.2.4]{THR},
we have a natural isomorphism of dihedral sets
\[
\Ndi (P^\gp \times Q^\gp) \times_{P^\gp \times Q^\gp} (P\times Q)
\simeq
(\Ndi P^\gp \times \Ndi Q^\gp) \times_{P^\gp \times Q^\gp} (P\times Q).
\]
The left-hand side is isomorphic to $\Ndrep (P\times Q)$,
and the right-hand side is isomorphic to $\Ndrep P \times \Ndrep Q$.
\end{proof}

For a commutative monoid $P$ with involution,
let $N^\sigma P$ denote the \emph{real nerve of $P$},
see \cite[Definition 4.2.1]{THR}.

\begin{prop}
\label{dih.25}
Let $P$ be a commutative monoid with involution.
Then there is a natural isomorphism of dihedral sets
\[
\Ndrep P
\simeq
P\times \Nsigma P^\gp.
\]
\end{prop}
\begin{proof}
Immediate from \cite[Proposition 4.2.6]{THR} and \eqref{dih.13.1}.
\end{proof}

\begin{df}
\label{dih.14}
For a commutative monoid $P$ with involution $w$,
let $(i_\sharp i^*P)^{ex}$ denote the commutative monoid $P\oplus P^\gp$ with the involution $w$ given by
\[
w(x,y):=(w(x),w(x)-w(y)).
\]
We have the commutative triangle
\[
\begin{tikzcd}
& (i_\sharp i^*P)^{ex}\ar[rd,"\mu^{ex}"]
\\
i_\sharp i^*P\ar[ru,"\gamma"]\ar[rr,"\mu"]&
&
P
\end{tikzcd}
\]
such that $\mu(x,y):=x+y$, $\mu^{ex}(x,y):=x$, and $\gamma(x,y):=(x+y,y)$ for $x,y\in P$.
We note that $\mu$ is the counit homomorphism.
The construction of $(i_\sharp i^*P)^{ex}$ is an equivariant analogue of the exactification in \cite[Proposition I.4.2.19]{Ogu}.
\end{df}

\begin{prop}
\label{dih.15}
Let $P$ be a commutative monoid with involution.
Then there is a natural equivalence of $\Ztwo$-spectra
\[
\Sphere[\Bdrep P]
\simeq
\Sphere[P]\wedge_{\Sphere[(i_\sharp i^*P)^{ex}]} \Sphere[P].
\]
\end{prop}
\begin{proof}
Let $\rE P^\gp$ denote the total simplicial set of $P^\gp$,
and let $Q$ be the real simplicial set whose underlying simplicial set is $P\times \rE P^\gp$,
and whose involution is given by
\[
(x,g_0,\ldots,g_q)\in P\times (P^\gp)^{\times(q+1)}
\mapsto
(w(x),w(x)-w(g_q),\ldots,w(x)-w(g_0))
\]
in simplicial degree $q$ for every integer $q\geq 0$.
Consider the maps of real simplicial sets
\[
(i_\sharp i^* P)^{ex} \to Q \to P\times \Nsigma P^\gp
\]
given by
$
(x,g)
\mapsto
(x,g,\ldots,g)
$
and
$
(x,g_0,\ldots,g_q)
\mapsto
(x,g_1-g_0,\ldots,g_q-g_{q-1})
$
in simplicial degree $q$.

There is a real simplicial isomorphism
\[
Q \oplus_{(i_\sharp i^*P)^{ex}} P
\simeq
P\times \Nsigma P^\gp.
\]
Since $Q$ is degreewise the disjoint union of finite copies of $(i_\sharp i^*P)^{ex}$ and $(i_\sharp i^*P)^{ex}\amalg (i_\sharp i^*P)^{ex}$ with the switching involution as an $(i_\sharp i^*P)^{ex}$-set, Lemma \ref{drep.6} yields a natural equivalence of $\Ztwo$-spectra
\begin{equation}
\label{dih.15.1}
\Sphere[Q] \wedge_{\Sphere[(i_\sharp i^*P)^{ex}]} \Sphere[P]
\simeq
\Sphere[P\times \Nsigma P^\gp].
\end{equation}

Let us show that the map $Q\to P$ given by
\[
(x,g_0,\ldots,g_q)\mapsto x
\]
in simplicial degree $q$ is a $\Ztwo$-homotopy equivalence. Its underlying map of simplicial sets is the projection $P\times \rE P^\gp \to P$, which is a homotopy equivalence.
The $\Ztwo$-fixed point of the Segal subdivision $\sd_{\sigma} Q$ is in simplicial degree $q$ the set
\[
\{
(x,g_0,\ldots,g_q,w(x)-w(g_q),\ldots,w(x)-w(g_0))
:
x\in P^{\Ztwo},g_0,\ldots,g_q\in P^\gp
\}.
\]
From this description,
we see that the induced map
\[
(\sd_{\sigma} Q)^{\Ztwo}
\to
(\sd_{\sigma} P)^{\Ztwo}
\]
can be identified with the projection $P\times \rE P^\gp \to P$.
Hence the map $Q\to P$ is a $\Ztwo$-homotopy equivalence.
Combine this with Proposition \ref{dih.25} and \eqref{dih.15.1} to obtain the desired equivalence.
\end{proof}

\begin{rmk}
For a commutative monoid $P$ with involution,
we can regard $\Sphere[\Bdi P]$ and $\Sphere[\Bdrep P]$ as objects of $\NAlg^{\Ztwo}$ by Propositions \ref{drep.2} and \ref{dih.15}.
\end{rmk}

\begin{prop}
\label{drep.3}
Let $P$ be a commutative monoid with involution.
Then there are natural equivalences of $\Ztwo$-spectra
\[
i^*\Sphere[\Bdi P]
\simeq
\Sphere[\Bcy i^*P]
\text{ and }
i^*\Sphere[\Bdrep P]
\simeq
\Sphere[\Brep i^*P].
\]
\end{prop}
\begin{proof}
Observe that we obtain $\Bcy i^*P$ (resp.\ $\Brep i^*P$) by forgetting the involution structure on $\Bdi P$ (resp.\ $\Bdrep P$).
\end{proof}

\section{THR of log rings with involutions}
\label{thrring}

For $R\in \NAlg^{\Ztwo}$,
the \emph{real topological Hochschild homology of $R$} is defined to be
\[
\THR(R)
:=
R\wedge_{N^{\Ztwo}i^*R}R,
\]
where the maps $N^{\Ztwo}i^*R \to R$ in the formulation are the counit of the adjunction pair $(N^{\Ztwo},i^*)$.
If $A$ is a commutative ring with involution,
then
the equivariant Eilenberg-MacLane spectrum $\EM A$ can be realized as a commutative orthogonal $\Z/2$-ring spectrum by \cite[Example 11.12]{Schwede},
so we have $\EM A\in \NAlg^{\Z/2}$ together with \cite[Remark A.2.3]{THR}.
We set $\THR(A):=\THR(\EM A)$.
For a commutative monoid $P$ with involution,
there is a natural equivalence of $\Ztwo$-spectra
\begin{equation}
\THR(\Sphere[P])
\simeq
\Sphere[\Bdi P],
\end{equation}
see \cite{Hog} and also  \cite[Proposition 5.9]{DMPR21}.
One can also show this using Proposition \ref{drep.2} and the natural equivalence of $\Ztwo$-spectra $N^{\Ztwo} i^*\Sphere[P]\simeq \Sphere[i_\sharp i^*P]$.

Recall from \cite[Definition III.1.2.3]{Ogu} that a \emph{log ring}\footnote{A log ring is often called a \emph{pre-log ring} in the literature, e.g.,  \cite[Definition 2.1]{Rog09}.}  $(A,P)$ is a commutative ring $A$ equipped with a homomorphism $P\to A$ of commutative monoids,
where the monoid operation on $A$ is the multiplication.
A \emph{homomorphism of log rings $(A,P)\to (B,Q)$} is a pair of homomorphisms $A\to B$ and $P\to Q$ such that the square
\[
\begin{tikzcd}
P\ar[d]\ar[r]&
A\ar[d]
\\
Q\ar[r]&
B
\end{tikzcd}
\]
commutes.
We regard a commutative ring $A$ as a log ring $(A,\{1\})$.

Rognes \cite[Definition 8.11, Remark 8.12]{Rog09} 
defined the \emph{topological Hochschild homology of a log ring $(A,P)$} as the coproduct in $\CAlg$
\begin{equation}
\THH(A,P)
:=
\THH(A)\wedge_{\Sphere[\Bcy P]}\Sphere[\Brep P],
\end{equation}
where the map $\Sphere[\Bcy P]\simeq \THH(\Sphere[P])\to \THH(A)$ in the formulation is obtained by applying $\THH$ to the induced map $\Sphere[P]\to \EM A$.

\begin{df}
\label{dih.22}
The \emph{real topological Hochschild homology of a log ring $(A,P)$ with involution} is the coproduct in $\NAlg^{\Ztwo}$
\begin{equation}
\label{dih.22.1}
\THR(A,P)
:=
\THR(A)\wedge_{\Sphere[\Bdi P]}\Sphere[\Bdrep P],
\end{equation}
where the map $\Sphere[\Bdi P]\simeq \THR(\Sphere[P])\to \THR(A)$ in this formulation is obtained by applying $\THR$ to the induced map $\Sphere[P]\to \EM A$.
\end{df}

We obviously have $\THR(A,\{1\})\simeq\THR(A)$ for every commutative ring $A$.

\begin{df}
Let $P\to M$ be a homomorphism of commutative monoids with involutions.
For notational convenience,
we set
\[
\THR(\Sphere[M],P)
:=
\Sphere[\Bdi M]\wedge_{\Sphere[\Bdi P]}\Sphere[\Bdrep P].
\]
\end{df}

\begin{prop}
\label{thrlog.8}
Let $Q\to M$ be a homomorphism of commutative monoids with involutions.
Then there is a natural equivalence of $\Ztwo$-spectra
\[
\THR(A[M],P\oplus Q)
\simeq
\THR(A,P) \wedge \THR(\Sphere[M],Q).
\]
\end{prop}
\begin{proof}
By \cite[Proposition 4.2.4]{THR} and Proposition \ref{drep.4},
there are natural isomorphisms of dihedral sets
\[
\Ndi (P\oplus Q)
\simeq
\Ndi P \times \Ndi Q
\text{ and }
\Ndrep (P\oplus Q)
\simeq
\Ndrep P \times \Ndrep Q.
\]
Apply $\Sphere[-]$ to these,
and use \eqref{dih.22.1} to obtain the desired equivalence.
\end{proof}

Let $S^\sigma$ be $S^1$ with the involution given by $e^{i\theta}\in S^1 \mapsto e^{-i\theta}$.

\begin{exm}
\label{thrlog.10}
By \cite[(4.12)]{THR},
we have an equivalence of $\Ztwo$-spectra
\[
\THR(\Sphere[\N],\N)
\simeq
\bigoplus_{d=0}^\infty \Sphere[S^\sigma].
\]
Furthermore,
\cite[Propositions 4.2.11, 4.2.12]{THR} implies that the induced map
\[
\THR(\Sphere[\N])
\to
\THR(\Sphere[\N],\N)
\]
can be written as the componentwise induced map
\[
\Sphere\oplus \bigoplus_{d=1}^\infty \Sphere[S^\sigma]
\to
\bigoplus_{d=0}^\infty \Sphere[S^\sigma].
\]
\end{exm}

In the remaining part of this section,
we investigate how $\THH$ and $\THR$ interact with the functors $i^*$ and $i_*$.

\begin{prop}
\label{thrlog.1}
Let $(A,P)$ be a log ring with involution.
Then there is a natural equivalence of spectra
\[
i^*\THR(A,P)
\simeq
\THH(i^*A,i^*P).
\]
\end{prop}
\begin{proof}
We have a natural equivalence $i^*\THR(A)\simeq \THH(i^*A)$
by \cite[Proposition 3.4.7]{THR}.
Since $i^*\colon \NAlg^{\Ztwo}\to \CAlg$ preserves colimits,
we have a natural equivalence
\[
i^*\THR(A,P)
\simeq
i^*\THR(A)\wedge_{i^*\Sphere[\Bdi P]}i^*\Sphere[\Bdrep P].
\]
Together with Proposition \ref{drep.3},
we obtain the desired equivalence.
\end{proof}

For a commutative monoid $P$,
let $\ul{P}$ denote the commutative monoid $P$ with the trivial involution.

\begin{const}
\label{thrlog.11}
Let $(A,P)$ be a log ring.
We have the composite map of spectra
\[
i^*\THR(i_*A,\ul{P})
\xrightarrow{\simeq}
\THH(A\oplus A,P)
\to
\THH(A,P),
\]
where log structure homomorphism $\ul{P}\to i_*A$ sends $p\in P$ to
$(\alpha(p),\alpha(p))$,
where $\alpha\colon P\to A$ is the log structure homomorphism.
and the first arrow is obtained by Proposition \ref{thrlog.1}, and the second arrow is induced by the summation homomorphism $A\oplus A\to A$.
By adjunction,
we obtain a map of $\Z/2$-spectra
\begin{equation}
\label{thrlog.11.1}
\THR(i_*A,\ul{P})
\to
i_*\THH(A,P).
\end{equation}
\end{const}

\begin{prop}
\label{thrlog.3}
Let $(A,P)$ be a log ring.
Then \eqref{thrlog.11.1} is an equivalence.
\end{prop}
\begin{proof}
Lemma \ref{thrlog.6} and Proposition \ref{drep.3}
yield natural equivalences of $\Ztwo$-spectra
\begin{align*}
i_*\THH(A,P)
\simeq &
i_*(\THH(A)\wedge_{i^*\Sphere[\Bdi \ul{P}]}i^*\Sphere[\Bdrep \ul{P}])
\\
\simeq &
i_*\THH(A)\wedge_{\Sphere[\Bdi \ul{P}]}\Sphere[\Bdrep \ul{P}].
\end{align*}
By \cite[Propositions 2.1.4, 2.3.3]{THR},
we have a natural equivalence of $\Ztwo$-spectra
\[
i_*\THH(A)
\simeq
\THR(i_*A).
\]
Combine these with the definition of $\THR(i_*A,\ul{P})$ to conclude.
\end{proof}

\begin{prop}
\label{thrlog.4}
The functor $\THR$ from the category of log rings with involutions to $\Z/2$-spectra preserves filtered colimits.
\end{prop}
\begin{proof}
One can directly check that the endofunctors $P\mapsto i_\sharp i^*P, (i_\sharp i^*P)^{ex}$ on the category of commutative monoids with involutions preserve colimits.
By Propositions \ref{drep.2} and \ref{dih.15},
the functors $P\mapsto \Sphere[\Bdi P],\Sphere[\Bdrep P]$ from the category of commutative monoids with involutions to $\NAlg_{\Z/2}$ preserve colimits.

On the other hand,
the functor $A\mapsto \THR(A)$ from the category of commutative rings with involutions to $\NAlg_{\Z/2}$ preserves filtered colimits since the Eilenberg-MacLane functor preserves filtered colimits, $N^{\Z/2}$ and $i^*$ preserve colimits, and $\wedge$ is the coproduct.
It follows that the functor $(A,P)\mapsto \THR(A,P)$ from the category of commutative rings with involutions to $\NAlg_{\Z/2}$ preserves filtered colimits.
The forgetful functor $\NAlg_{\Z/2}\to \Sp_{\Z/2}$ preserves filtered colimits as observed in \cite[\S 1]{THR},
which finishes the proof.
\end{proof}

\section{THR of log schemes with involutions}
\label{thrscheme}

So far,
we have discussed $\THR$ of log rings $(A,P)$ with involutions.
The purpose of this section is to define $\THR(X)$ for every separated log scheme $X$ and to show that a canonical map
\[
\THR(A,P)
\to
\THR(\Spec(A,P))
\]
is an equivalence of $\Z/2$-spectra.

Let us briefly review basic notation and terminology in log geometry.
We refer to Ogus's book \cite{Ogu} for the details.
For a commutative monoid $P$,
let $P^*$ denote its submonoid of units.
We set $\ol{P}:=P/P^*$.
We say that $P$ is \emph{integral} if the induced homomorphism $P\to P^\gp$ is injective.
\begin{itemize}
\item
For a log scheme $X$,
let $\ul{X}$ be its underlying scheme,
and let $\cM_X$ be its structure sheaf of monoids.
\item
A morphism of log schemes $f\colon Y\to X$ is \emph{strict} if the induced morphism $Y\to X\times_{\ul{X}}\ul{Y}$ is an isomorphism.
\item
For a log ring $(A,P)$,
let $\Spec(A,P)$ denote its associated log scheme in the sense of \cite[Definition III.1.2.3]{Ogu}.
\item
For a commutative monoid $P$,
we set $\A_P:=\Spec(\Z[P],P)$,
whose structure homomorphism $P\to \Z[P]$ sends $p\in P$ to $p$.
\item
A \emph{chart $P$ of a log scheme $X$} is a commutative monoid $P$ together with a strict morphism $X\to \A_P$ of log schemes.
\item
A log scheme $X$ is \emph{integral} if $\cM_X$ is a sheaf of integral monoids.
\item
A log scheme $X$ is \emph{fine saturated} (or simply \emph{fs}) if $X$ admits strict \'etale locally a chart $P$ such that $P$ is a fine saturated monoid.
\end{itemize}

Let $G$ be a finite group.
For a point $x$ of a $G$-scheme $X$,
the \emph{scheme-theoretic stabilizer of $X$ at $x$} is
\[
G_x
:=
\ker(
\{
g\in G:gx=x
\}\to \mathrm{Aut}(k(x))).
\]
A morphism of $G$-schemes $Y\to X$ is an \emph{isovariant \'etale cover} if the induced homomorphism $G_y\to G_{f(x)}$ is an isomorphism for every point $y\in Y$
and $f$ is \'etale and surjective after forgetting the $G$-action.

For a commutative ring $A$ with involution, an \emph{$A$-algebra with involution} is a commutative ring $B$ with involution equipped with a $\Z/2$-equivariant map $A\to B$.

\begin{prop}
\label{descent.7}
Let $(A,P)$ be a log ring with involution.
Then the presheaf $\THR(-,P)$ on the opposite category of $A$-algebras with involution is an isovariant \'etale hypersheaf.
\end{prop}
\begin{proof}
Let $B\to C$ be an isovariant \'etale homomorphism of $A$-algebras with involutions.
By \cite[Theorem 3.2.3]{THR},
there are natural equivalences of $\Ztwo$-spectra
\[
\THR(B,P)
\wedge_{\EM B}
\EM C
\simeq
\THR(B,P)
\wedge_{\EM \iota (B^{\Ztwo})}
\EM \iota (C^{\Ztwo})
\simeq
\THR(C,P),
\]
where $\iota M$ denotes the commutative monoid with the trivial involution for a commutative monoid $M$.
Argue as in the proof of \cite[Theorem 3.4.3]{THR} to show the claim.
\end{proof}

\begin{lem}
\label{strict.2}
Let $\theta\colon P\to Q$ be a homomorphism of integral monoids.
If $\ol{\theta}\colon \ol{P}\to \ol{Q}$ is an isomorphism, then the induced homomorphism of monoids
\begin{equation}
\label{strict.2.1}
\eta
\colon
P\oplus_{P^*}Q^*
\to
Q
\end{equation}
is an isomorphism.
\end{lem}
\begin{proof}
Since $\ol{\theta}$ is an isomorphism,
$\eta$ is surjective.
To show that $\eta$ is injective,
assume $\eta(p,v)=\eta(p',v')$ with $p,p'\in P$ and $v,v'\in Q^*$.
This implies $\theta(p)+v=\theta(p')+v'$.
Since $\ol{\theta}$ is an isomorphism,
there exists $u\in P^*$ satisfying $p=p'+u$.
Together with the assumption that $Q$ is integral,
we have $v'=v+\theta(u)$.
Use \cite[Proposition I.1.1.5(3)]{Ogu} to see that $(p,v)=(p',v')$ in $P\oplus_{P^*} Q^*$.
Hence $\eta$ is an isomorphism.
\end{proof}

\begin{lem}
\label{strict.3}
Let $\theta\colon P\to Q$ be a homomorphism of integral monoids with involution.
If $\ol{\theta}\colon \ol{P}\to \ol{Q}$ is surjective,
then the induced map
\[
\THR(\Sphere[Q],P)
\to
\THR(\Sphere[Q],Q)
\]
is an equivalence of $\Ztwo$-spectra.
\end{lem}
\begin{proof}
Due to Propositions \ref{drep.2} and \ref{dih.15},
it suffices to show that the induced map
\[
(\Sphere[Q] \wedge_{\Sphere[i_\sharp i^*Q]} \Sphere[Q])
\wedge_{(\Sphere[P] \wedge_{\Sphere[i_\sharp i^*P]} \Sphere[P])}
(\Sphere[P] \wedge_{\Sphere[(i_\sharp i^*P)^{ex}]} \Sphere[P])
\to
\Sphere[Q] \wedge_{\Sphere[(i_\sharp i^*Q)^{ex}]} \Sphere[Q]
\]
is an equivalence of $\Ztwo$-spectra.
The left-hand side is equivalent to
\[
\Sphere[Q]\wedge_{(\Sphere[i_\sharp i^*Q]\wedge_{\Sphere[i_\sharp i^*P]}\Sphere[(i_\sharp i^*P)^{ex}])} \Sphere[Q].
\]
Hence it suffices to show that the induced map
\[
\Sphere[i_\sharp i^*Q]\wedge_{\Sphere[i_\sharp i^*P]}\Sphere[(i_\sharp i^*P)^{ex}]
\to
\Sphere[(i_\sharp i^*Q)^{ex}]
\]
is an equivalence of $\Ztwo$-spectra.
Since $P$ is integral,
$u+x=x$ implies $u=0$ for $u\in P^*$ and $x\in P$.
It follows that
$i^*P$ is a free $i^*P^*$-set,
so the induced map
\[
\Sphere[i_\sharp i^* Q^*]
\wedge_{\Sphere[i_\sharp i^* P^*]}
\Sphere[i_\sharp i^* P]
\to
\Sphere[i_\sharp i^*Q]
\]
is an equivalence of $\Z/2$-spectra.
Hence it suffices to show that the induced map
\begin{equation}
\label{strict.3.1}
\Sphere[i_\sharp i^*Q^*]\wedge_{\Sphere[i_\sharp i^*P^*]}\Sphere[(i_\sharp i^*P)^{ex}]
\to
\Sphere[(i_\sharp i^*Q)^{ex}]
\end{equation}
is an equivalence of $\Ztwo$-spectra.

Consider the localization $P_F$ with respect to the face $F:=\theta^{-1}(Q^*)$ of $P$.
By \cite[Theorem I.4.5.7, Proposition I.4.6.3(3), Remark I.4.6.6]{Ogu},
$i^*P_F$ is a filtered colimit of free $i^*P$-sets.
This implies that $\Sphere[i_\sharp i^*P_F]$ is a filtered colimit of free $\Sphere[i_\sharp i^*P]$-modules.
It follows that the induced map
\[
\Sphere[i_\sharp i^*P_F]\wedge_{\Sphere[i_\sharp i^*P]}\Sphere[(i_\sharp i^*P)^{ex}]
\to
\Sphere[(i_\sharp i^*P_F)^{ex}]
\]
is an equivalence.
Replace $P\to Q$ with $P_F\to Q$ to reduce to the case when $\ol{\theta}$ is an isomorphism.
Then we have $Q\simeq P\oplus_{P^*} Q^*$ by Proposition \ref{strict.2}.

We have the induced cocartesian square
\begin{equation}
\label{strict.3.2}
\begin{tikzcd}
i^*P^*\oplus i^*P^*\ar[r,"\alpha"]\ar[d]&
i^*P\oplus i^*P^\gp\ar[d]
\\
i^*Q^*\oplus i^*Q^*\ar[r]&
i^*Q\oplus i^*Q^\gp,
\end{tikzcd}
\end{equation}
where $\alpha$ sends $(p,p')$ to $(p+p',p')$.
Observe that $i^*P\oplus i^*P^\gp$ is a free $i^*P^*\oplus i^*P^*$-set.
Apply $\Sphere[-]$ to this square to see that the induced map
\[
\Sphere[i^*Q \oplus i^*Q]
\wedge_{\Sphere[i^*P\oplus i^*P]}
\Sphere[i^*P\oplus i^*P^\gp]
\to
\Sphere[i^*Q\oplus i^*Q^\gp].
\]
is an equivalence of spectra,
i.e.,
\eqref{strict.3.1} becomes an equivalence after applying $i^*$.
Hence it remains to show that \eqref{strict.3.1} becomes an equivalence after applying $\Phi^{\Z/2}$.

We claim that $\theta$ is exact,
i.e.,
the induced square
\[
\begin{tikzcd}
P\ar[d,"\theta"']\ar[r,"\eta"]&
P^\gp\ar[d,"\theta"]
\\
Q\ar[r,"\eta'"]&
Q^\gp
\end{tikzcd}
\]
is cartesian.
Assume that $\eta'(q)=\theta(p')$ for some $q\in Q$ and $p'\in P^\gp$.
Using \cite[Proposition I.1.1.5(3)]{Ogu} for $Q\simeq P\oplus_{P^*}Q^*$,
we can write $q=(p,v)$ for some $p\in P$ and $v\in Q^*$.
Since the group completion functor is a left adjoint,
we have an isomorphism $Q^\gp \simeq P^\gp\oplus_{P^*}Q^*$.
Using \cite[Proposition I.1.1.5(3)]{Ogu} for this,
we have $(\eta(p),v)=(p',0)$ in $Q'^\gp$,
and hence we have $v=\theta(u)$ for some $u\in P^*$.
This implies that $q=\theta(p+u)$ and $p'=\eta(p+u)$,
so $\theta$ is exact.

Observe that the $\Ztwo$-fixed point monoids of $i_\sharp i^* P^*$ and $i_\sharp i^*Q^*$ are isomorphic to $i^*P^*$ and $i^*Q^*$.
Let $w$ denote the involutions on $P$ and $Q$,
and let $M$ and $N$ be the $\Ztwo$-fixed point monoids of $(i_\sharp i^*P)^{ex}$ and $(i_\sharp i^*Q)^{ex}$.
Observe that $N$ is a submonoid of $Q\oplus Q^\gp$ consisting of $(q,y)$ such that $\eta'(q)=y+w(y)$.
We have a similar description for $M$ too.
Using $Q^\gp \simeq P^\gp\oplus_{P^*}Q^*$,
we can find $x\in P^\gp$ and $v\in Q^*$ such that $y=\theta(x)+v$.
Since $\theta$ is exact,
there exists a unique $p\in P$ such that $\eta(p)=x+w(x)$ and $\theta(p)=q-v-w(v)$.
This implies that the induced homomorphism $M\oplus_{i^*P^*}i^*Q^*\to N$ is surjective.

On the other hand,
if $(p,x)\in M$ and $v\in Q^*$ satisfies $(\theta(p),\theta(x))=(v+w(v),v)$ in $N$,
then we have $p\in \theta^{-1}(Q^*)$.
By \cite[Proposition I.1.1.5(3)]{Ogu} for $Q\cong P\oplus_{P^*}Q^*$,
we have $\theta^{-1}(Q^*)=P^*$, so we have $p\in P^*$.
This implies that the induced homomorphism $M\oplus_{i^*P^*}i^*Q^*\to N$ is injective and hence an isomorphism.
Since $M$ is a free $i^*P^*$-set,
we deduce that the induced map
\[
\Sphere[i^*Q]\wedge_{\Sphere[i^*P]}\Sphere[M]
\to
\Sphere[N]
\]
is an equivalence of spectra,
i.e.,
\eqref{strict.3.1} becomes an equivalence after applying $\Phi^{\Ztwo}$.
\end{proof}

A morphism of log schemes with involutions $f\colon Y\to X$ is a \emph{strict isovariant \'etale cover} if $f$ is strict and its underlying morphism of schemes with involutions $\ul{f}\colon \ul{Y}\to \ul{X}$ is an isovariant \'etale cover.
The \emph{strict isovariant \'etale topology} is the topology generated by strict isovariant \'etale covers.
Let $siso\et$ be the shorthand for this topology.

\begin{prop}
\label{siso.4}
Let $X$ be a separated integral log scheme with involution.
Then there exists a strict isovariant \'etale covering $\{U_i\to X\}_{i\in I}$ such that each $\ul{U_i}$ is an affine scheme with involution.
\end{prop}
\begin{proof}
Let $w$ be the involution on $X$.
Choose an open neighborhood $U_x$ of $x$ in $X$ such that $\ul{U_x}$ is an affine scheme.
If $x=w(x)$,
then the underlying scheme of $U_x\cap w(U_x)$ is an affine scheme with involution since $X$ is separated.
If $x\neq w(x)$,
then the underlying scheme of $U_x\amalg w(U_x)$ is an affine scheme with involution.
We finish the proof using \cite[Corollary 2.19]{HeKO}.
\end{proof}

Let $\lSch$ denote the category of separated integral log schemes.
Following Definition \ref{thrlog.9},
we obtain the category
$\lSch_{\Ztwo}$,
and we have the forgetful functor $i^*\colon \lSch_{\Ztwo}\to \lSch$ with a left adjoint $i_\sharp$.

\begin{df}
\label{descent.4}
Let $X$ be a separated integral log scheme with involution.
Consider the presheaf $\THR|_X$ of $\Z/2$-spectra
given by
\[
\THR|_X(U)
:=\THR(\Gamma(U,\cO_U),\Gamma(U,\cM_U))
\]
for every strict isovariant \'etale morphism $U\to X$.
The \emph{real topological Hochschild homology of $X$} is
\[
\THR(X)
:=
(L_{siso\et} \THR|_X)(X)\in \Sp^{\Z/2},
\]
where $L_{siso\et}\colon \PSh(\lSch_{\Z/2},\Sp^{\Z/2})\to \Sh_{siso\et}(\lSch_{\Z/2},\Sp^{\Z/2})$ denotes the strict isovariant \'etale sheafification functor, see \cite[Lemma 1.3.4.3]{SAG} for sheafification functors.

Observe that $\THR$ is a strict isovariant \'etale sheaf by definition.
\end{df}

\begin{thm}
\label{descent.5}
Let $(A,P)$ be an integral log ring with involution.
Then the induced map of $\Z/2$-spectra
\[
\THR(A,P)
\to
\THR(\Spec(A,P))
\]
is an equivalence.
\end{thm}
\begin{proof}
Consider the presheaf $\THR|_{(A,P)}$ of $\Z/2$-spectra given by
\[
\THR|_{(A,P)}(B):= \THR(B,P)
\]
for every $A$-algebra $B$ with involution such that $\Spec(A)\to \Spec(B)$ is isovariant \'etale.
We have the induced map of presheaves
\[
\THR|_{(A,P)}
\to
\THR|_{\Spec(A,P)}.
\]
Observe that the left-hand side is an isovariant \'etale sheaf by Proposition \ref{descent.7}.
It suffices to show that this map of presheaves becomes an equivalence after sheafification
since taking the global sections produces the desired equivalence.

For this,
it suffices to show that the map of stalks is an equivalence.
For every sheaf $\cF$ and a point $x$ of $X$,
the stalk $\cF_x$ is given by the filtered colimit $\colim_{U\ni x}\cF(U)$.
Hence by Propositions \ref{thrlog.4} and \ref{siso.4},
we reduce to the case when $A$ is a local ring with involution.
In this case,
we need to show that the induced map
\[
\THR(A,P)
\to
\THR(A,P^a)
\]
is an equivalence,
where $P^a$ is the logification of $P$,
which is given by $P^a:=P\oplus_{\theta^{-1}(A^*)}A^*$ if $\theta\colon P\to A$ is the structure map.
This is a consequence of Lemma \ref{strict.3} for $P\to P^a$.
\end{proof}

In the next results,
we explain how $\THR$ is related with $\THH$ under the functors $i^*$ and $i_*$.
We review the definition of $\THH$ of schemes in \cite[Definition 8.3.7]{logSH} as follows.
We consider the presheaf of spectra
$\THH|_X$ on $X_{\et}$ given by
\[
\THH|_X(U):=\THH(\Gamma(U,\cO_U),\Gamma(U,\cM_U))
\]
for $U\in X_{\et}$.
The topological Hochschild homology of $X$ is
\[
\THH(X):=
(L_{\et}\THH|_X)(X)\in \Sp,
\]
where $L_{\et}$ denotes the \'etale sheafification functor.
The induced map
\[
\THH(A,P)
\simeq
\THH(\Spec(A,P))
\]
is an equivalence of spectra for every integral log ring $(A,P)$ as Theorem \ref{descent.5}.

\begin{prop}
\label{siso.7}
Let $X$ be a separated integral log scheme with involution.
Then there is a natural equivalence of spectra
\[
\THH(i^*X)
\simeq
i^*\THR(X).
\]
\end{prop}
\begin{proof}
By Proposition \ref{siso.4},
we reduce to the case when $X=\Spec(A,P)$ for some log ring $(A,P)$ with involution.
Proposition \ref{thrlog.1} finishes the proof.
\end{proof}

\begin{prop}
\label{descent.12}
Let $X$ be a separated integral log scheme.
Then there is a natural equivalence of $\Ztwo$-spectra
\[
\THR(i_\sharp X)
\simeq
i_*\THH(X)
\]
\end{prop}
\begin{proof}
As above,
we reduce to the case when $X=\Spec(A,P)$ for some log ring $(A,P)$.
In this case,
there is a natural isomorphism
\[
i_\sharp \Spec(A,P)
\simeq
\Spec(i_* A,\ul{P}).
\]
Proposition \ref{thrlog.3} finishes the proof.
\end{proof}

\begin{cor}
\label{descent.13}
Let $X$ be a separated integral log scheme with involution.
Then there is a natural equivalence of spectra
\[
\THR(i_\sharp X)^{\Z/2}
\simeq
\THH(X).
\]
\end{cor}
\begin{proof}
Combine Proposition \ref{descent.12} with \cite[Proposition A.2.7(2)]{THR} to conclude.
\end{proof}

\section{TCR of log schemes with involutions}
\label{TCR}

Throughout this section,
$p$ is a prime number.
This section is relying on  \cite{QS1} and \cite{QS2} due to Quigley-Shah,
which we review as follows.
Keep in mind that our $C_p$ and $\Ztwo$ correspond to their $\mu_p$ and $C_2$.
See \cite[\S 3]{synTCR} for another review.

A \emph{$\Z/2$-$\infty$-category} (resp.\ \emph{$\Z/2$-$\infty$-space}) is a presheaf of $\infty$-categories (resp.\ spaces) on the orbit category $\cO_{\Z/2}$.
Recall that $\cO_{\Z/2}$ can be described as the diagram
\[
\begin{tikzcd}
(\Z/2)/e\ar[loop left,"w"]\ar[r,"\mathrm{res}"]&
(\Z/2)/(\Z/2)
\end{tikzcd}
\]
Observe that we can naturally view a $\Z/2$-space as a $\Z/2$-category.
A $\Z/2$-functor of $\Z/2$-categories is a morphism of presheaves.
For $\Z/2$-categories $\cC$ and $\cD$,
let $\Fun_{\Z/2}(\cC,\cD)$ be the $\infty$-category of $\Z/2$-functors $\cC\to \cD$.

Consider the $\Z/2$-$\infty$-categories $\ul{\Sp}^{\Z/2}$ and $\ul{\NAlg}^{\Z/2}$ in \cite[Example 3.2]{synTCR}.
We have $\ul{\Sp}^{\Z/2}((\Z/2)/e)=\Sp$, $\ul{\Sp}^{\Z/2}((\Z/2)/(\Z/2))=\Sp^{\Z/2}$, and $\ul{\Sp}^{\Z/2}(\mathrm{res})=i^*$.
We have a similar description for $\NAlg^{\Z/2}$ too.
Let $BS^\sigma$ and $BC_p^\sigma$ be the classifying $\Z/2$-spaces of $S^\sigma$ and $C_p^\sigma$,
where $C_p^\sigma$ denotes the $\Z/2$-subspace of $S^\sigma$ whose underlying space is $C_p$.
We have the $\infty$-categories
\begin{gather*}
(\Sp^{\Z/2})^{BS^\sigma}
:=
\Fun_{\Z/2}(
BS^\sigma,
\ul{\Sp}^{\Z/2}
),
\\
(\Sp^{\Z/2})^{BC_p^\sigma}
:=
\Fun_{\Z/2}(
BC_p^\sigma,
\ul{\Sp}^{\Z/2}
).
\end{gather*}
We similarly have the $\infty$-categories $(\NAlg^{\Z/2})^{BS^\sigma}$ and $(\NAlg^{\Z/2})^{BC_p^\sigma}$.

For equivariant homotopy theory,
we refer to \cite[\S 3]{MR3838307} for the model categorical approach (the group $G$ can be a compact Lie group) and \cite[\S 9]{BH21} for the $\infty$-categorical approach ($G$ is only a profinite group).
For a compact Lie group $G$,
let $\Sp^G$ be the $\infty$-category of $G$-spectra,
which is the underlying $\infty$-category of orthogonal $G$-spectra.
According to \cite[Theorem A, Remark 1.8]{QS1},
we have the forgetful functors
\begin{gather*}
j^*\colon \Sp^{O(2)}
\to
(\Sp^{\Z/2})^{BS^\sigma},
\\
j^*\colon \Sp^{D_{2p}}
\to
(\Sp^{\Z/2})^{BC_p^\sigma},
\end{gather*}
which admit right adjoints $j_*$.
We have the fixed point functors
\begin{gather*}
(-)^{C_p}
\colon
\Sp^{O(2)}
\to
\Sp^{O(2)/C_p}
\simeq
\Sp^{O(2)},
\\
(-)^{C_p}
\colon
\Sp^{D_{2p}}
\to
\Sp^{\Z/2}.
\end{gather*}
We have the norm functor
\[
N_{\Z/2}^{D_{2p}}
\colon
\Sp^{\Z/2}\to \Sp^{D_{2p}}.
\]
We have the geometric fixed point functors
\begin{gather*}
\Phi^{C_p}
\colon
\Sp^{O(2)}
\to
\Sp^{O(2)/C_p}
\simeq
\Sp^{O(2)},
\\
\Phi^{C_p}
\colon
\Sp^{D_{2p}}
\to
\Sp^{\Ztwo}.
\end{gather*}
Recall from \cite[Definition 1.6, Remark 1.8]{QS1} that the \emph{parametrized Tate constructions} are
\begin{gather*}
(-)^{tC_p^\sigma}
:=
j^* \Phi^{C_p} j_*
\colon
(\Sp^{\Z/2})^{BS^\sigma}
\to
(\Sp^{\Z/2})^{BS^\sigma},
\\
(-)^{tC_p^\sigma}
:=
\Phi^{C_p} j_*
\colon
(\Sp^{\Z/2})^{BC_p^\sigma}
\to
\Sp^{\Z/2}.
\end{gather*}
For $X\in \Sp^{\Z/2}$, we often use the notation $X^{\otimes C_p^\sigma}:=j^*N_{\Z/2}^{D_{2p}} X$.
We have the parametrized Tate diagonal given by the composite natural map
\[
\Delta
\colon
X \xrightarrow{\simeq} \Phi^{C_p} N_{\Z/2}^{D_{2p}} X \to \Phi^{C_p} j_*j^* N_{\Z/2}^{D_{2p}} X = (X^{\otimes C_p^\sigma})^{tC_p^\sigma}.
\]

A \emph{real cyclotomic spectrum $X$} is an object of $(\Sp^{\Z/2})^{BS^\sigma}$ equipped with maps
\[
\varphi_p
\colon
X
\to
X^{tC_p^\sigma}
\]
in $(\Sp^{\Z/2})^{BS^\sigma}$ for all primes $p$.
In \cite[Definition 1.20]{QS2},
the $\infty$-category of real cyclotomic spectra is defined to be the lax equalizer
\[
\begin{tikzcd}[column sep=large]
\RCycSp
:=
\LEq((\Sp^{\Z/2})^{BS^\sigma}
\arrow[r,shift left=0.5ex,"\id"]\ar[r,shift right,"\prod_p (-)^{tC_p^\sigma}"']&
\prod_p (\Sp^{\Z/2})^{BS^\sigma}).
\end{tikzcd}
\]
See \cite[Remark 1.21]{QS2} for the $\Z/2$-$\infty$-category of real cyclotomic spectra $\ul{\RCycSp}$.
We refer to \cite[Remark 4.3]{QS2} for the functor
\[
\TCR\colon \RCycSp\to \Sp^{\Z/2}.
\]

By \cite[Proposition 3.7]{synTCR},
the forgetful functor
\[
q^*
\colon
(\NAlg^{\Z/2})^{BS^\sigma}
\to
\NAlg^{\Ztwo}
\]
admits a left adjoint $q_\otimes$.
For $A\in \NAlg^{\Ztwo}$,
following \cite[\S 5]{QS2} (see also \cite[Definition 3.11, Proposition 3.12]{synTCR},
we set
\[
\THR(A):=
q_\otimes A.
\]
Recall from \cite[\S 5]{QS2} the following facts:
(1) There exists a unique map
\[
\varphi_p
\colon
\THR(A)
\to
\THR(A)^{tC_p^\sigma}
\]
such that the induced square
\begin{equation}
\label{TCR.0.1}
\begin{tikzcd}
A\ar[d,"\Delta"']\ar[r]&
\THR(A)\ar[d,"\varphi_p"]
\\
(A^{\otimes C_p^\sigma})^{tC_p^\sigma}\ar[r]&
\THR(A)^{tC_p^\sigma}
\end{tikzcd}
\end{equation}
commutes,
where the lower horizontal map is induced by the inclusion $C_p\to S^1$.
(2) $\varphi_p$ is a map in $\NAlg^{\Z/2}$.
(3) As a consequence of (2),
$\THR(A)$ is an object of $\NAlg(\RCycSp)$,
where $\NAlg(\RCycSp)$ denotes the $\infty$-category of normed algebras in $\ul{\RCycSp}$ defined as in \cite[Definition 9.14]{BH21}.
We write down the definition as follows.
For $X\in \FinGpd$,
let $\Fin_X$ denote the category of finite coverings of $X$ in $\FinGpd$.
We have a functor
\[
\ul{\RCycSp}^\otimes
\colon
\Span(\Fin_{B(\Z/2)})
\to
\Cat_\infty
\]
in \cite[Point 3 after Remark 4.3]{QS2}.
A \emph{normed algebra in $\ul{\RCycSp}$} is a section of $\ul{\RCycSp}^\otimes$ over $\Span(\Fin_{B(\Z/2)})$ that is cocartesian over the backward morphisms in $\Span(\Fin_{B(\Z/2)})$. Recall that a backward morphism in a span $\infty$-category is a morphism of the form $X\leftarrow Y \xrightarrow{\id} Y$, see \cite[Appendix C]{BH21}.

\begin{prop}
\label{TCR.1}
Let $P$ be a commutative monoid with involution.
Then for every prime $p$,
there exists a natural equivalence of $\Z/2$-spaces
\begin{equation}
\label{TCR.1.1}
(\Bdi P)^{C_p}
\simeq
\Bdi P.
\end{equation}
\end{prop}
\begin{proof}
Let $\sd_p$ be the $p$-fold subdivision functor in \cite{BHM93}.
For every integer $q\geq 0$,
the set of $q$-simplicies in $(\sd_p \Ndi P)^{C_p}$ is the set
\[
\{
(x_0,\ldots,x_q,\ldots,x_0,\ldots,x_q):
x_0,\ldots,x_q\in P
\},
\]
which is isomorphic to the set of $q$-simplicies in $\Ndi P$.
This isomorphism is also compatible with the involutions,
so we obtain the desired equivalence.
\end{proof}

\begin{prop}
\label{TCR.2}
Let $P$ be a commutative monoid with involution.
Then for every prime $p$,
there exists a natural equivalence of $\Z/2$-spaces
\begin{equation}
\label{TCR.2.1}
(\Bdrep P)^{C_p}
\simeq
\Bdrep P.
\end{equation}
\end{prop}
\begin{proof}
The isomorphism $\Ndi P^\gp \cong (\sd_p \Ndi P^\gp)^{C_p}$ obtained by the proof of Proposition \ref{TCR.1} can be restricted to an isomorphism $\Ndrep P \cong (\sd_p \Ndrep P)^{C_p}$.
This produces the desired equivalence.
\end{proof}

\begin{prop}
\label{TCR.3}
Let $P$ be a commutative monoid with involution.
Then the Frobenius $\varphi_p \colon \THR(\Sphere[P])\to \THR(\Sphere[P])^{tC_p^\sigma}$ is equivalent to the composite
\begin{equation}
\label{TCR.3.1}
\Sphere[\Bdi P]
\xrightarrow{\simeq}
\Sphere[(\Bdi P)^{C_p}]
\xrightarrow{\simeq}
\Phi^{C_p}\Sphere[\Bdi P]
\to
\Sphere[\Bdi P]^{tC_p^\sigma},
\end{equation}
where the first map is obtained by \eqref{TCR.1.1}.
\end{prop}
\begin{proof}
Consider the commutative square
\[
\begin{tikzcd}
P\ar[r]\ar[d,leftarrow,"\simeq"']&
\Bdi P\ar[d,leftarrow,"\simeq"]
\\
(P^{\oplus C_p})^{C_P}\ar[r]&
(\Bdi P)^{C_p},
\end{tikzcd}
\]
where the right vertical arrow is obtained by \eqref{TCR.1.1}, and the lower horizontal arrow is induced by the inclusion $C_p\to S^1$.
After taking $\Sphere[-]$,
we obtain the commutative square
\[
\begin{tikzcd}
\Sphere[P]\ar[d,leftarrow,"\simeq"']\ar[r]&
\Sphere[\Bdi P]\ar[d,leftarrow,"\simeq"]
\\
\Phi^{C_p} N_{\Ztwo}^{D_{2p}} \Sphere[P]\ar[r]&
\Phi^{C_p} \Sphere[\Bdi P].
\end{tikzcd}
\]
Using the natural transformation $\Phi^{C_p}\to \Phi^{C_p}j_*j^*$,
we obtain the commutative square
\[
\begin{tikzcd}
\Sphere[P]\ar[d]\ar[r]&
\Sphere[\Bdi P]\ar[d]
\\
(\Sphere[P]^{\otimes C_p^\sigma})^{tC_p^\sigma}\ar[r]&
\Sphere[\Bdi P]^{tC_p^\sigma}.
\end{tikzcd}
\]
Compare this with \eqref{TCR.0.1} to conclude.
\end{proof}

\begin{df}
\label{TCR.4}
Let $P$ be a commutative monoid with involution.
We have the Frobenius $\varphi\colon \Sphere[\Bdrep P]\to \Sphere[\Bdrep P]^{tC_p^\sigma}$ given by the composite
\begin{equation}
\label{TCR.4.1}
\Sphere[\Bdrep P]
\xrightarrow{\simeq}
\Sphere[(\Bdrep P)^{C_p}]
\xrightarrow{\simeq}
\Phi^{C_p}\Sphere[\Bdrep P]
\to
\Sphere[\Bdrep P]^{tC_p^\sigma},
\end{equation}
where the first map is obtained by \eqref{TCR.2.1}.
Since the canonical map $\Bdi P\to \Bdrep P$ is a map of commutative monoids in the category of topological $O(2)$-spaces,
we obtain a natural map
\begin{equation}
\label{TCR.4.2}
\Sphere[\Bdi P]\to \Sphere[\Bdrep P]
\end{equation}
in $(\NAlg^{\Z/2})^{BS^\sigma}$.
Compare \eqref{TCR.3.1} and \eqref{TCR.4.1} to see that \eqref{TCR.4.2} is compatible with $\varphi_p$.
Hence we can promote \eqref{TCR.4.2} to a map in 
in $\NAlg(\RCycSp)$.

Let $(A,P)$ be a log ring with involution.
We take the coproduct
\[
\THR(A,P)
:=
\THR(A)
\wedge_{\Sphere[\Bdi P]}
\Sphere[\Bdrep P]
\]
in $\NAlg(\RCycSp)$.
The \emph{real topological cyclic homology of $(A,P)$} is
\[
\TCR(A,P)
:=
\TCR(\THR(A,P)).
\]
\end{df}

\begin{prop}
\label{TCR.5}
The forgetful functor
\[
\RCycSp
\to
\Sp^{\Ztwo}
\]
is conservative, exact, symmetric monoidal, and preserves colimits and finite limits.
\end{prop}
\begin{proof}
We refer to \cite[Remark 4.3]{QS2}.
\end{proof}

\begin{df}
A morphism of log schemes with involutions is a \emph{strict equivariant Nisnevich cover} if it is strict and its underlying morphism of schemes with involutions is an equivariant Nisnevich cover in the sense of Voevodsky \cite[\S 3.1]{zbMATH05645855}.
See also \cite[\S 2]{HeKO}.

The \emph{strict equivariant Nisnevich topology} is the topology generated by strict equivariant Nisnevich covers.
Let $seNis$ be the shorthand for this topology.
Observe that the strict equivariant Nisnevich topology is coarser than the strict isovariant \'etale topology.
\end{df}

\begin{prop}
\label{TCR.7}
Let $(A,P)$ be a finite dimensional noetherian log ring with involution.
Then the presheaf $\THR(-,P)$ on the opposite category of $A$-algebras of finite type with involution is an equivariant Nisnevich sheaf of real cyclotomic spectra.
\end{prop}
\begin{proof}
For an equivariant Nisnevich distinguished square $Q$ (see \cite[\S 2.1]{HeKO}) consisting of $A$-algebras of finite type with involution, we need to show that $\THR(Q,P)$ is a cocartesian square of real cyclotomic spectra.
By Proposition \ref{TCR.5},
it suffices to show that $\THR(Q,P)$ is a cocartesian square of $\Z/2$-spectra.
This is a consequence of Proposition \ref{descent.7}.
\end{proof}

\begin{rmk}
We do not know whether $\THR(-,P)$ in Proposition \ref{TCR.7} is an isovariant \'etale sheaf of real cyclotomic spectra.
An affirmative answer would remove the finite dimensional noetherian assumption in Theorem \ref{TCR.10} below.
\end{rmk}

\begin{prop}
\label{TCR.8}
Let $X$ be a finite dimensional noetherian separated log scheme with involution.
Then there exists an equivariant Nisnevich covering $\{U_i\to X\}_{i\in I}$ such that each $\ul{U_i}$ is an affine scheme with involution.
\end{prop}
\begin{proof}
This is a consequence of \cite[Lemma 2.20]{HeKO}.
\end{proof}

\begin{prop}
\label{TCR.11}
The functor $\THR$ from the category of log rings with involutions to $\Z/2$-spectra preserves filtered colimits.
\end{prop}
\begin{proof}
The is a consequence of Propositions \ref{thrlog.4} and \ref{TCR.5}.
\end{proof}

\begin{df}
\label{TCR.9}
Let $X$ be a finite dimensional noetherian separated integral log scheme with involution.
Consider the presheaf $\THR|_X$ of real cyclotomic spectra
given by
\[
\THR|_X(U)
:=\THR(\Gamma(U,\cO_U),\Gamma(U,\cM_U))
\]
for every strict isovariant \'etale morphism $U\to X$.
Consider
\[
\THR(X)
:=
(L_{seNis} \THR|_X)(X)\in \RCycSp,
\]
where $L_{seNis}$ denotes the strict equivariant Nisnevich sheafification functor.

Observe that $\THR$ is a strict equivariant Nisnevich sheaf of real cyclotomic spectra by definition.
Furthermore,
if we forget the real cyclotomic structure on $\THR(X)$,
then we recover $\THR(X)$ in Definition \ref{descent.4} by Theorem \ref{TCR.10} below.

The \emph{real topological cyclic homology of $X$} is
\[
\TCR(X):=\TCR(\THR(X)).
\]
\end{df}

\begin{thm}
\label{TCR.10}
Let $(A,P)$ be a finite dimensional noetherian integral log ring with involution.
Then the induced map
\[
\THR(A,P)
\to
\THR(\Spec(A,P))
\]
is an equivalence of real cyclotomic spectra.
In particular,
the induced map
\[
\TCR(A,P)
\to
\TCR(\Spec(A,P))
\]
is an equivalence of $\Z/2$-spectra.
\end{thm}
\begin{proof}
Argue as in the proof of Theorem \ref{descent.5},
but use instead Propositions \ref{TCR.7}, \ref{TCR.8}, and \ref{TCR.11} to reduce to the case when $A$ is a local ring.
In this case,
we need to show that the induced map
\[
\THR(A,P)
\to
\THR(A,P^a)
\]
is an equivalence of real cyclotomic spectra,
where $P^a:=P\oplus_{\theta^{-1}(A^*)}A^*$ if $\theta\colon P\to A$ is the structure map.
This is a consequence of Theorem \ref{descent.5} and Proposition \ref{TCR.5}.
\end{proof}

The purpose of the remaining part of this section is to show Proposition \ref{TCR.24},
which is needed for Proposition \ref{mot.14}.
For this,
we will check that various squares commute.

The map of $\Z/2$-spaces $q\colon *\to BS^\sigma$ induces the forgetful functor
\[
q^*\colon (\Sp^{\Z/2})^{BS^\sigma}\to \Sp^{\Z/2},
\]
which is conservative by \cite[Proposition 3.8]{synTCR} and admits a left adjoint $q_\sharp$ and right adjoint $q_*$ by \cite[Proposition 3.7]{synTCR}.

\begin{prop}
\label{TCR.13}
There are natural equivalences
\begin{gather}
\label{TCR.13.1}
q^*q_\sharp\simeq \Sigma^\infty S_+^\sigma \wedge (-),
\\
\label{TCR.13.2}
q^*q_*\simeq \Sigma^{-\sigma} \Sigma^\infty S_+^\sigma \wedge (-).
\end{gather}
\end{prop}
\begin{proof}
For \eqref{TCR.13.1},
consider the cartesian square
\[
\begin{tikzcd}
S^\sigma\ar[d,"r"']\ar[r,"r"]&
*\ar[d,"q"]
\\
*\ar[r,"q"]&
BS^\sigma
\end{tikzcd}
\]
where $r\colon S^\sigma\to *$ is the unique map.
By \cite[Lemma 4.3]{QS1},
we have a natural equivalence $q^*q_\sharp\simeq r_\sharp r^*$.
To obtain \eqref{TCR.13.1},
observe that the projection formula \cite[Lemma 5.44]{QS1} yields a natural equivalence $r_\sharp r^*\simeq \Sigma^\infty S_+^\sigma \wedge (-)$.

We have $\fib(\id \to q^*q_\sharp)\simeq \Sigma^\sigma$.
By taking right adjoints,
we have $\cofib(q^*q_*\to \id)\simeq \Sigma^{-\sigma}$,
which yields \eqref{TCR.13.2}.
\end{proof}

\begin{rmk}
\label{TCR.14}
We similarly have the conservative forgetful functor
\[
q^*\colon \Sp^{BS^1}\to \Sp
\]
with a left adjoint $q_\sharp$ and right adjoint $q_*$.
There are also natural equivalences
\begin{gather*}
q^*q_\sharp\simeq \Sigma^\infty S_+^1 \wedge (-),
\\
q^*q_*\simeq \Sigma^{-1} \Sigma^\infty S_+^1 \wedge (-).
\end{gather*}
\end{rmk}

\begin{prop}
\label{TCR.15}
There are induced commutative squares
\[
\begin{tikzcd}
\Sp^{\Z/2}\ar[d,"i^*"']\ar[r,"q_\sharp"]&
(\Sp^{\Z/2})^{BS^\sigma}\ar[d,"i^*"]
\\
\Sp\ar[r,"q_\sharp"]&
\Sp^{BS^1},
\end{tikzcd}
\quad
\begin{tikzcd}
\Sp^{\Z/2}\ar[d,"i^*"']\ar[r,"q_*"]&
(\Sp^{\Z/2})^{BS^\sigma}\ar[d,"i^*"]
\\
\Sp\ar[r,"q_*"]&
\Sp^{BS^1}.
\end{tikzcd}
\]
\end{prop}
\begin{proof}
We focus on the first square since the proofs are similar.
Since $q^*\colon \Sp^{BS^1}\to \Sp$ is conservative,
it suffices to show that the composite natural transformation
\[
q^*q_\sharp i^*
\to
q^*i^*q_\sharp
\xrightarrow{\simeq}
i^*q^*q_\sharp
\]
is an isomorphism.
Using Proposition \ref{TCR.13} and Remark \ref{TCR.14},
this natural transformation can be identified with the natural transformation
\[
\Sigma^\infty S_+^1\wedge i^*
\to
i^*(\Sigma^\infty S_+^\sigma \wedge (-)).
\]
This is an equivalence since $i^*\Sigma^\infty S^\sigma\simeq \Sigma^\infty S^1$.
\end{proof}

Let $v\colon BS^\sigma \to B(S^\sigma/C_p^\sigma)\simeq BS^\sigma$ be the functor induced by the quotient map $S^\sigma \to S^\sigma/C_p^\sigma$.
We have the induced functor $v^*\colon (\Sp^{\Z/2})^{BS^\sigma}\to (\Sp^{\Z/2})^{BS^\sigma}$ whose left adjoint is $(-)_{hC_p^\sigma}$ and right adjoint is $(-)^{hC_p^\sigma}$,
which correspond to $(-)_{h_{C_2}S^1}$ and $(-)^{h_{C_2}S^1}$ used in \cite[Example 5.57]{QS1}.
Similarly,
we have the induced functor $v^*\colon \Sp^{BS^1}\to \Sp^{BS^1}$ whose left adjoint is $(-)_{hC_p}$ and right adjoint is $(-)^{hC_p}$,
which are used in \cite[Theorem 1.3]{NS}.

\begin{prop}
\label{TCR.16}
There are induced commutative squares
\[
\begin{tikzcd}
\Sp^{BS^1}\ar[r,"v^*"]\ar[d,"i_\sharp"']&
\Sp^{BS^1}\ar[d,"i_\sharp"]
\\
(\Sp^{\Z/2})^{BS^\sigma}\ar[r,"v^*"]&
(\Sp^{\Z/2})^{BS^\sigma},
\end{tikzcd}
\quad
\begin{tikzcd}
\Sp^{BS^1}\ar[r,"v^*"]\ar[d,"i_*"']&
\Sp^{BS^1}\ar[d,"i_*"]
\\
(\Sp^{\Z/2})^{BS^\sigma}\ar[r,"v^*"]&
(\Sp^{\Z/2})^{BS^\sigma}.
\end{tikzcd}
\]
\end{prop}
\begin{proof}
We focus on the first square since the proofs are similar.
We have the induced commutative diagram
\[
\begin{tikzcd}
i_\sharp q^*v^*\ar[d,"\simeq"']\ar[r]&
q^*i_\sharp v^*\ar[r]&
q^*v^*i_\sharp
\ar[d,"\simeq"]
\\
i_\sharp q^*\ar[rr]&
&
q^*i_\sharp.
\end{tikzcd}
\]
The vertical arrows are equivalences since the composite $*\xrightarrow{q} BS^\sigma \xrightarrow{v} BS^\sigma$ agrees with $q$ and the same claim holds for $BS^1$ too.
The lower horizontal and upper right horizontal arrows are equivalences by Proposition \ref{TCR.15}.
Hence the upper left horizontal arrow is an equivalence.
To conclude,
use the fact that $q^*\colon \Sp^{BS^1}\to \Sp$ is conservative.
\end{proof}

Let $u\colon BS^\sigma \to *$ be the unique map.
We have the induced functor $u^*\colon \Sp^{\Z/2}\to (\Sp^{\Z/2})^{BS^\sigma}$ whose left adjoint is $(-)_{hS^\sigma}$ and right adjoint is $(-)^{hS^1}$.
Similarly,
we have the induced functor $u^*\colon \Sp\to \Sp^{BS^1}$ whose left adjoint is $(-)_{hS^1}$ and right adjoint is $(-)^{hS^1}$.

\begin{prop}
\label{TCR.22}
There are induced commutative squares
\[
\begin{tikzcd}
\Sp\ar[r,"u^*"]\ar[d,"i_\sharp"']&
\Sp^{BS^1}\ar[d,"i_\sharp"]
\\
\Sp^{\Z/2}\ar[r,"u^*"]&
(\Sp^{\Z/2})^{BS^\sigma},
\end{tikzcd}
\quad
\begin{tikzcd}
\Sp\ar[r,"u^*"]\ar[d,"i_*"']&
\Sp^{BS^1}\ar[d,"i_*"]
\\
\Sp^{\Z/2}\ar[r,"u^*"]&
(\Sp^{\Z/2})^{BS^\sigma}.
\end{tikzcd}
\]
\end{prop}
\begin{proof}
Argue as in Proposition \ref{TCR.16}.
\end{proof}

\begin{prop}
\label{TCR.28}
There are induced commutative squares
\[
\begin{tikzcd}
(\Sp^{\Z/2})^{BS^\sigma}\ar[r,"(-)^{tC_p^\sigma}"]\ar[d,"i^*"']&
(\Sp^{\Z/2})^{BS^\sigma}\ar[d,"i^*"]
\\
\Sp^{BS^1}\ar[r,"(-)^{tC_p}"]&
\Sp^{BS^1},
\end{tikzcd}
\quad
\begin{tikzcd}
(\Sp^{\Z/2})^{BS^\sigma}\ar[r,"(-)^{tS^\sigma}"]\ar[d,"i^*"']&
\Sp^{\Z/2}\ar[d,"i^*"]
\\
\Sp^{BS^1}\ar[r,"(-)^{tS^1}"]&
\Sp.
\end{tikzcd}
\]
\end{prop}
\begin{proof}
Consider the appropriate adjoint squares of the squares in Proposition \ref{TCR.16},
and compare the cofiber sequences $(-)_{hC_p}\to (-)^{hC_p}\to (-)^{tC_p}$ and $(-)_{hC_p^\sigma} \to (-)^{hC_p^\sigma} \to (-)^{tC_p^\sigma}$ to obtain the left square.
Argue similarly for the right square,
but use Proposition \ref{TCR.22} instead.
\end{proof}

\begin{const}
\label{TCR.17}
The motivic purity transformation \cite{Ayo07} can be adapted to the equivariant homotopy theory as follows.
Let $f\colon X\to S$ be a finite covering in $\FinGpd$.
Consider the induced diagram
\[
\begin{tikzcd}
X
\ar[r,"a"]&
X\times_S X\ar[d,"p_1"']\ar[r,"p_2"]&
X\ar[d,"f"]
\\
&
X\ar[r,"f"]&
S,
\end{tikzcd}
\]
where $a$ is the diagonal morphism, and $p_1$ (resp.\ $p_2$) is the first (resp.\ second) projection.
We set $\Sigma_f:=p_{2\sharp}a_*$.
We have the natural transformation 
\[
f_\sharp \xrightarrow{\mathfrak{p}_f} f_* \Sigma_f \colon \SH(X)\to \SH(S)
\]
given by the composite
\[
f_\sharp \xrightarrow{\simeq} f_\sharp p_{1*}a_*\to f_*p_{2\sharp }a_*.
\]
\end{const}

Recall that $i\colon *\to B(\Z/2)$ denotes the unique morphism.

\begin{prop}
\label{TCR.18}
The functor $\Sigma_i$ is equivalent to the identity functor,
and the natural transformation $\mathfrak{p}_i\colon i_\sharp \to i_*\Sigma_i$ is an equivalence.
\end{prop}
\begin{proof}
Let $r\colon \Z/2\to *$ be the unique map of finite groupoids,
and let $a\colon *\to \Z/2$ be the inclusion to the first point.
The functor $a_*\colon \Sp\to \Sp\times \Sp$ is $(\id,0)$,
and the functor $r_\sharp \colon \Sp\times \Sp\to \Sp$ is the direct sum functor.
Hence $\Sigma_i:=r_\sharp a_*$ is equivalent to the identity functor.

The pair of functors $(i^*,\Phi^{\Z/2})$ is conservative.
Since $\Phi^{\Z/2}i_\sharp \simeq \Phi^{\Z/2}i_*\simeq 0$ by \cite[Proposition A.2.7(3),(5)]{THR},
it suffices to show that $i^*\mathfrak{p}_i$ is an equivalence.
Using Proposition \ref{TCR.15},
it suffices to show that $\mathfrak{p}_r$ is an equivalence.
This can be shown directly using the following description in \cite[Proposition A.1.5]{THR}:
For a finite set $X$ with $n$ elements, $\SH(X)$ is equivalent to the product of $n$ copies of $\Sp$.
\end{proof}

\begin{const}
\label{TCR.19}
Consider the composite functor
\begin{equation}
\label{TCR.19.2}
\Sigma_i\colon \Sp^{BS^1} \xrightarrow{a_*} \Sp^{BS^1} \times \Sp^{BS^1} \xrightarrow{r_\sharp} \Sp^{BS^1},
\end{equation}
where $a_*:=(\id,0)$, and $r_\sharp$ is the direct sum functor.
As in Construction \ref{TCR.17},
we have the natural transformation
\begin{equation}
\label{TCR.19.1}
i_\sharp \xrightarrow{\mathfrak{p}_f} i_* \Sigma_i
\colon
(\Sp^{\Z/2})^{BS^\sigma}
\to
\Sp^{BS^1}
\end{equation}
\end{const}

\begin{prop}
\label{TCR.20}
The functor \eqref{TCR.19.2} is equivalent to the identity functor,
and the natural transformation \eqref{TCR.19.1} is an equivalence.
\end{prop}
\begin{proof}
The descriptions of $a_*$ and $r_\sharp$ show that \eqref{TCR.19.2} is equivalent to the identity functor.
Since $q^*\colon \Sp^{BS^1}\to \Sp$ is conservative,
we can reduce the second claim to Proposition \ref{TCR.18} by Proposition \ref{TCR.15}.
\end{proof}

\begin{prop}
\label{TCR.21}
There are induced commutative squares
\[
\begin{tikzcd}
\Sp^{BS^1}\ar[r,"(-)_{hC_p}"]\ar[d,"i_\sharp"']&
\Sp^{BS^1}\ar[d,"i_\sharp"]
\\
(\Sp^{\Z/2})^{BS^\sigma}
\ar[r,"(-)_{hC_p^\sigma}"]&
(\Sp^{\Z/2})^{BS^\sigma},
\end{tikzcd}
\quad
\begin{tikzcd}
\Sp^{BS^1}\ar[r,"(-)^{hC_p}"]\ar[d,"i_*"']&
\Sp^{BS^1}\ar[d,"i_*"]
\\
(\Sp^{\Z/2})^{BS^\sigma}
\ar[r,"(-)^{hC_p^\sigma}"]&
\Sp^{\Z/2},
\end{tikzcd}
\]
\[
\begin{tikzcd}
\Sp^{BS^1}\ar[r,"(-)_{hS^1}"]\ar[d,"i_\sharp"']&
\Sp\ar[d,"i_\sharp"],
\\
(\Sp^{\Z/2})^{BS^\sigma}
\ar[r,"(-)_{hS^\sigma}"]&
\Sp^{\Z/2},
\end{tikzcd}
\quad
\begin{tikzcd}
\Sp^{BS^1}\ar[r,"(-)^{hS^1}"]\ar[d,"i_*"']&
\Sp\ar[d,"i_*"],
\\
(\Sp^{\Z/2})^{BS^\sigma}
\ar[r,"(-)^{hS^\sigma}"]&
\Sp^{\Z/2}.
\end{tikzcd}
\]
\end{prop}
\begin{proof}
The right squares commute since their left adjoint squares commute.
Together with Propositions \ref{TCR.18} and \ref{TCR.20},
we see that the left squares commute too.
\end{proof}

\begin{prop}
\label{TCR.29}
There are induced commutative squares
\[
\begin{tikzcd}
\Sp^{BS^1}\ar[r,"(-)^{tC_p}"]\ar[d,"i_\sharp"']&
\Sp^{BS^1}\ar[d,"i_\sharp"]
\\
(\Sp^{\Z/2})^{BS^\sigma}
\ar[r,"(-)^{tC_p^\sigma}"]&
(\Sp^{\Z/2})^{BS^\sigma},
\end{tikzcd}
\quad
\begin{tikzcd}
\Sp^{BS^1}\ar[r,"(-)^{tC_p}"]\ar[d,"i_*"']&
\Sp^{BS^1}\ar[d,"i_*"]
\\
(\Sp^{\Z/2})^{BS^\sigma}
\ar[r,"(-)^{tC_p^\sigma}"]&
(\Sp^{\Z/2})^{BS^\sigma}.
\end{tikzcd}
\]
\end{prop}
\begin{proof}
Argue as in Proposition \ref{TCR.28}, but use Proposition \ref{TCR.21} instead.
\end{proof}

\begin{prop}
\label{TCR.12}
There is an adjunction
\[
i^*: \RCycSp \rightleftarrows \CycSp : i_*
\]
satisfying the following properties:
\begin{enumerate}
\item For $Y\in \RCycSp$,
the Frobenius $i^*Y\to (i^*Y)^{tC_p}$ is identified with $i^*\varphi$ if $\varphi\colon Y\to Y^{tC_p^\sigma}$ is the Frobenius.
\item For $X\in \CycSp$, the Frobenius $i_*X \to (i_*X)^{tC_p^\sigma}$ is identified with $i_*\varphi$ if $\varphi\colon X\to X^{tC_p}$ is the Frobenius.
\item $i^*$ is symmetric monoidal.
\end{enumerate}
\end{prop}
\begin{proof}
The functor $i^*\colon\RCycSp\to \CycSp$ is obtained by taking lax equalizers to the rows of the diagram
\begin{equation}
\label{TCR.12.1}
\begin{tikzcd}[column sep=large]
(\Sp^{\Z/2})^{BS^\sigma}\ar[r,shift left=0.5ex,"\id"]\ar[r,shift right=0.5ex,"\prod_p (-)^{tC_p^\sigma}"']\ar[d,"i^*"']&
\prod_p(\Sp^{\Z/2})^{BS^\sigma}\ar[d,"i^*"]
\\
\Sp^{BS^1}\ar[r,shift left=0.5ex,"\id"]\ar[r,shift right=0.5ex,"\prod_p (-)^{tC_p}"']&
\prod_p\Sp^{BS^1},
\end{tikzcd}
\end{equation}
whose two squares commute by Proposition \ref{TCR.28}.
The two squares in the induced diagram
\begin{equation}
\label{TCR.12.2}
\begin{tikzcd}[column sep=large]
\Sp^{BS^1}\ar[r,shift left=0.5ex,"\id"]\ar[r,shift right=0.5ex,"\prod_p (-)^{tC_p}"']\ar[d,"i_*"']&
\prod_p\Sp^{BS^1}\ar[d,"i_*"]
\\
(\Sp^{\Z/2})^{BS^\sigma}\ar[r,shift left=0.5ex,"\id"]\ar[r,shift right=0.5ex,"\prod_p (-)^{tC_p^\sigma}"']&
\prod_p(\Sp^{\Z/2})^{BS^\sigma}
\end{tikzcd}
\end{equation}
commute by Proposition \ref{TCR.29}.
Hence we see that a right adjoint $i_*$ is obtained by taking lax equalizers to the rows of \eqref{TCR.12.2} since we can check the counit-unit identities pointwise.
This implies the claims (1) and (2).

Recall from \cite[Construction IV.2.1(ii)]{NS} the following fact:
Let $\cC$, $\cD$, and $\cE$ be symmetric monoidal $\infty$-categories.
If $F\colon \cC\to \cD$ (resp.\ $G\colon \cC\to \cD$) is a symmetric (resp.\ lax symmetric) monoidal functor,
then having a symmetric monoidal functor $\cE\to \LEq(F,G)$ is equivalent to having a symmetric monoidal functor $H\colon \cE\to \cF$ and a lax symmetric monoidal transformation $F\circ H\to G\circ H$.

The canonical functor $\RCycSp\to (\Sp^{\Z/2})^{BS^\sigma}$ is symmetric monoidal by the above paragraph,
and one can show that the forgetful functor
$i^*\colon (\Sp^{\Z/2})^{BS^\sigma} \to \Sp^{BS^1}$ is symmetric monoidal using the descriptions of the symmetric monoidal structures on $\Sp^{BS^1}$ and $(\Sp^{\Z/2})^{BS^\sigma}$ obtained by \cite[Construction IV.2.1(1)]{NS} and \cite[Example 5.13]{QS1}.
Hence the composite functor $\RCycSp\to \Sp^{BS^1}$ is symmetric monoidal too.
Furthermore,
we have the lax symmetric natural transformation of the two composite functors
\[
\begin{tikzcd}
\RCycSp\ar[r]&
(\Sp^{\Z/2})^{BS^\sigma}\ar[r,"i^*"]&
\Sp^{\Z/2}\ar[r,shift left=0.5ex,"\id"]\ar[r,shift right=0.5ex,"\prod_p (-)^{tC_p}"']&
\prod_p\Sp^{BS^1}
\end{tikzcd}
\]
using \eqref{TCR.12.1} and the above paragraph.
We deduce the claim (3) using the above paragraph again.
\end{proof}

\begin{prop}
\label{TCR.23}
For $X\in \CycSp$ and $Y\in \RCycSp$,
there are natural equivalences
\begin{gather*}
i^*\TCR(Y)
\simeq
\TC(i^*Y),
\\
i_*\TC(X)
\simeq
\TCR(i_*X).
\end{gather*}
\end{prop}
\begin{proof}
Using the descriptions of $i^*Y\to (i^*Y)^{tC_p}$ and $i_*X\to (i_*X)^{tC_p^\sigma}$ in Proposition \ref{TCR.12},
the claim follows from \cite[Corollary 1.5]{NS}, and \cite[Theorem C(3)]{QS2},
and Propositions \ref{TCR.22} and \ref{TCR.21}.
\end{proof}

\begin{prop}
\label{TCR.27}
Let $(A,P)$ be a log ring with involution.
Then there is a natural equivalence of cyclotomic spectra
\[
i^*\THR(A,P)
\simeq
\THH(i^*A,i^*P)
\]
and hence an equivalence of spectra
\[
i^*\TCR(A,P)
\simeq
\TCR(i^*A,i^*P).
\]
\end{prop}
\begin{proof}
Recall that we have a natural equivalence of spectra $i^*\THR(A)\simeq \THH(i^*A)$ by \cite[Proposition 3.4.7]{THR}.
Apply $i^*$ to \eqref{TCR.0.1},
and compare this with the square in \cite[p.\ 342]{NS} to show that $i^*\THR(A)\xrightarrow{i^*\varphi_p} i^*(\THR(A)^{tC_p^\sigma})$ can be identified with $\THH(A)\xrightarrow{\varphi_p} \THH(A)^{tC_p}$.
We also need Proposition \ref{TCR.28} here.
Together with Proposition \ref{TCR.12}(2),
we have an equivalence of cyclotomic spectra $i^*\THR(A)\simeq \THH(A)$.

On the other hand, if we apply $i^*$ to \eqref{TCR.3.1} and \eqref{TCR.4.1},
then we get $\varphi_p$ for $\Sphere[\Bcy P]$ and $\Sphere[\Brep P]$ by Proposition \ref{drep.3} by Proposition \ref{TCR.28}.
Together with Proposition \ref{TCR.12}(2),
we have equivalences of cyclotomic spectra $i^*\Sphere[\Bdi P]\simeq \Sphere[\Bcy P]$ and $i^*\Sphere[\Bdrep P]\simeq \Sphere[\Brep P]$.
Proposition \ref{TCR.12}(3) finishes the proof.
\end{proof}

\begin{prop}
\label{TCR.26}
There is a commutative square
\[
\begin{tikzcd}
\CycSp\ar[d]\ar[r,"i_*"]&
\RCycSp\ar[d]
\\
\Sp\ar[r,"i_*"]&
\Sp^{\Z/2},
\end{tikzcd}
\]
where the vertical arrows are the forgetful functors.
\end{prop}
\begin{proof}
By Proposition \ref{TCR.16},
it suffices to show that there is a commutative square
\[
\begin{tikzcd}
\CycSp\ar[d]\ar[r,"i_*"]&
\RCycSp\ar[d]
\\
\Sp^{BS^1}\ar[r,"i_*"]&
(\Sp^{\Z/2})^{BS^\sigma},
\end{tikzcd}
\]
where the vertical arrows are the forgetful functors.
This can be obtained using \eqref{TCR.12.2}.
\end{proof}

\begin{prop}
\label{TCR.25}
Let $(A,P)$ be a log ring.
Then there is a natural equivalence of real cyclotomic spectra
\[
\THR(i_*A,\ul{P})
\simeq
i_*\THH(A,P).
\]
\end{prop}
\begin{proof}
Argue as in Construction \ref{thrlog.11} but use Proposition \ref{TCR.27} instead to construct a natural map of real cyclotomic spectra $\THR(i_*A,\ul{P})\to i_*\THR(A,P)$.
By Proposition \ref{TCR.5},
it suffices to show that this becomes an equivalence of $\Z/2$-spectra after applying the forgetful functor $\RCycSp\to \Sp^{\Z/2}$.
This is a consequence of Propositions \ref{thrlog.3} and \ref{TCR.26}.
\end{proof}

\begin{prop}
\label{TCR.24}
Let $X$ be a finite dimensional noetherian separated integral log scheme.
Then there is a natural equivalence of real cyclotomic spectra
\[
\THR(i_\sharp X)
\simeq
i_*\THH(X)
\]
and hence an equivalence of $\Z/2$-spectra
\[
\TCR(i_\sharp X)
\simeq
i_*\TC(X)
\]
and an equivalence of spectra
\[
\TCR(i_\sharp X)^{\Z/2}
\simeq
\TC(X).
\]
\end{prop}
\begin{proof}
Argue as in Proposition \ref{descent.12}, but use Propositions \ref{TCR.8} and \ref{TCR.25} to obtain the first equivalence.
For the remaining equivalences,
use Proposition \ref{TCR.23} and \cite[Proposition A.2.7(2)]{THR}.
\end{proof}

We will use the notation
\[
\THR^{\Z/2}(-):=\THR(-)^{\Z/2},
\text{ }
\TCR^{\Z/2}(-):=\TCR(-)^{\Z/2}.
\]

\section{Motivic representability}
\label{mot}

Throughout this section,
we fix a finite dimensional noetherian separated scheme $S$ (with the trivial log structure and involution).
The purpose of this section is to represent $\THH^{\Z/2}$ and $\TCR^{\Z/2}$ in $\logSH(S)$ and $\THR$ and $\TCR$ in a $\Ztwo$-equivariant analogue of $\logSH(S)$.

We begin with recalling the ingredients for defining $\logSH(S)$.
Let $\SmlSm/S$ denote the category of fs log schemes $Y$ of finite type over $S$ such that $Y\to S$ is log smooth and $\ul{Y}\to S$ is smooth.
A morphism in $\SmlSm/S$ is a \emph{strict Nisnevich cover} if it is strict and its underlying morphism of schemes is a Nisnevich cover.
We have the $\infty$-category of strict Nisnevich sheaves of spectra $\Sh_{sNis}(\SmlSm/S,\Sp)$.

For a smooth scheme $X$ of finite type over $S$ with a strict normal crossing divisor $D$,
let $(X,D)$ denote the fs log scheme with the underlying scheme $X$ and the compactifying log structure associated with the open immersion $X-D\to X$.
By \cite[Lemma A.5.10]{logDM},
every object of $\SmlSm/S$ arises as this form.
We often regard $\P^{n-1}$ as a closed complement of $\A^n$ in $\P^n$,
and we form $(\P^n,\P^{n-1})$.
We set $\square:=(\P^1,\infty)$.

Now, the \emph{$\infty$-category of logarithmic motivic $S^1$-spectra} is defined to be the localization
\[
\logSH_{S^1}(S)
:=
(\P^\bullet,\P^{\bullet-1})^{-1}\Sh_{sNis}(\SmlSm/S,\Sp),
\]
where $(\P^\bullet,\P^{\bullet-1})^{-1}$ denotes the class of projections $(\P^n,\P^{n-1})\times X\to X$ for all $X\in \SmlSm/S$ and integers $n\geq 1$.
This is one of the various models of $\logSH_{S^1}(S)$ in \cite[\S 3.4]{logSH}.
The \emph{$\infty$-category of logarithmic motivic spectra} is defined to be the $\infty$-category of $\P^1$-spectra in $\logSH_{S^1}(S)$:
\[
\logSH(S)
:=
\Sp_{\P^1}(\logSH_{S^1}(S)).
\]

\begin{rmk}
\label{mot.12}
In this section,
we often compute $\THR(X\times Y)$ in terms of $\THR(X)$ for noetherian separated fs log schemes $X$ and $Y$ such that there is a Zariski covering $\{\Spec(\Z[M_i],P_i)\to Y\}_{i\in I}$ for some homomorphisms of monoids $P_i\to M_i$.
In view of Proposition \ref{thrlog.8},
the base $X$ is irrelevant for the computation.
We will often argue as if everything takes place on the spectral scheme $\Spec(\Sphere)$ instead of $X$ for notational convenience.
\end{rmk}

Recall that $S^\sigma$ is $S^1$ with the involution given by $e^{i\theta}\in S^1 \mapsto e^{-i\theta}$.
We have the functor
\[
\Sigma^\sigma
:=
\Sigma^\infty S^\sigma \wedge (-)
\colon
\Sp^{\Ztwo}\to \Sp^{\Ztwo}.
\]
Let us recall the computation of $\THR$ for the projective spaces with the trivial involutions as follows.

\begin{prop}
\label{mot.2}
Let $X$ be a noetherian separated fs log scheme.
Then there exists a natural equivalence of $\Ztwo$-spectra
\[
\THR(X\times \P^n)
\simeq
\left\{
\begin{array}{ll}
\THR(X) \oplus \bigoplus_{j=1}^{\lfloor n/2 \rfloor} i_*\THH(X) & \text{if $n$ is even},
\\
\THR(X) \oplus \bigoplus_{j=1}^{\lfloor n/2 \rfloor} i_*\THH(X) \oplus \Sigma^{n(\sigma-1)}\THR(X) & \text{if $n$ is odd}.
\end{array}
\right.
\]
\end{prop}
\begin{proof}
Following Remark \ref{mot.12},
we can argue as if everything takes place on $\Sphere$ instead of $X$.
Then the claim is due to \cite[Theorem 5.2.6]{THR}.
\end{proof}

Our next goal is to show that $\THR^{\Z/2}$ and $\TCR^{\Z/2}$ are representable in $\logSH_{S^1}(S)$.
For this,
we need the notion of cubes \cite[Definition 6.1.1.2]{HA},
which we recall as follows.

For a set $I$,
let $\bP(I)$ denote the set of subsets of $I$.
We impose the partially ordered set structure on $\bP(I)$ with respect to the inclusion,
and we regard $\bP(I)$ as the associated category.
For an $\infty$-category $\cC$,
an \emph{$I$-cube in $\cC$} is a functor
\[
Q
\colon
\bP(I)
\to
\cC.
\]
If the cardinal of $I$ is an integer $n$,
then an $I$-cube is called an \emph{$n$-cube}.
The \emph{total cofiber of $Q$} is
\[
\tcofib(Q)
:=
\cofib(\colim(Q|_{\bP(I)-\{I\}}) \to Q(I))
\]
whenever the colimit and cofiber exist.
The \emph{total fiber of $Q$} is
\[
\tfib(Q)
:=
\fib(Q(\emptyset)\to \lim(Q|_{\bP(I)-\{\emptyset\}}))
\]
whenever the limit and fiber exist.

\begin{thm}
\label{mot.1}
Let $X$ be a noetherian separated fs log scheme with involution.
Then the induced map
\[
\THR(X)
\to
\THR(X\times (\P^n,\P^{n-1}))
\]
is an equivalence of real cyclotomic spectra,
where we impose the trivial involution on $(\P^n,\P^{n-1})$.
\end{thm}
\begin{proof}
By Proposition \ref{TCR.5},
it suffices to show that the map is an equivalence of $\Z/2$-spectra.

Following Remark \ref{mot.12},
we will argue for $\Sphere$ instead of $X$.
Let $\ul{U_0},\ldots,\ul{U_n}$ be the standard cover of $\P_{\Sphere}^n$,
and we set $U_i:=(\P_{\Sphere}^n,\P_{\Sphere}^{n-1})\times_{\P_{\Sphere}^n} \ul{U_i}$.
For every nonempty subset $I\subset [n]$,
we set $U(I):=U_I:=\bigcap_{i\in I} U_i$.
We also set $U_\emptyset:=\Sphere$.
We regard $U$ as an $(n+1)$-cube.
We need to show $\tcofib(\THR(U))\simeq 0$.

For notational convenience,
we consider the commutative monoids
\begin{gather*}
P_j:=\{(x_1,\ldots,x_n)\in \Z^n : x_j\geq 0\}
\text{ for $j=1,\ldots,n$},
\\
P_0:=\{(x_1,\ldots,x_n)\in \Z^n : x_1+\ldots+x_n \leq 0\},
\\
F_0:=\{(x_1,\ldots,x_n)\in \Z^n : x_1+\ldots+x_n = 0\}.
\end{gather*}
For every subset $J\subset [n]$,
we set $P_J:=\bigcap_{j\in J}P_j$.
By convention,
we have 
$P_\emptyset:=\Z^n$.
There is a canonical equivalence
\[
\THR(U_I)
\simeq
\left\{
\begin{array}{ll}
\THR(\Sphere[P_{[n]-I}])& \text{if $0\in I$},
\\
\THR(\Sphere[P_{[n]-I}],P_{[n]-I}\cap F_0)
&\text{if $0\notin I$}.
\end{array}
\right.
\]
Each $P_J$ is isomorphic to products of finite copies of $\Z$ and $\N$.
The computations of $\Bdi \N$ and $\Bdi \Z$ in \cite[(4.12), Propositions 4.2.11, 4.2.12]{THR} and the fact that $\Bdi$ preserves finite products \cite[Proposition 4.2.4]{THR} yield a natural equivalence of $\Ztwo$-spectra
\[
\THR(U_I)
\simeq
\Sphere[\Phi(I)]
\]
with
\begin{gather*}
\Phi(I)
:=
\coprod_{(x_1,\ldots,x_n)\in \Z^n} \Phi(I;(x_1,\ldots,x_n)),
\\
\Phi(I;(x_1,\ldots,x_n))
:=
\left\{
\begin{array}{ll}
\Phi_1(I;x_1) \times \cdots \times \Phi_n(I;x_n)
&
\text{if $(x_1,\ldots,x_n)\in P_{[n]-I}$},
\\
\emptyset
&
\text{otherwise},
\end{array}
\right.
\\
\Phi_i(I;x_i)
:=
\left\{
\begin{array}{ll}
* & \text{if $x_i=0$ and $i\in I$},
\\
S^\sigma & \text{otherwise}.
\end{array}
\right.
\end{gather*}
Since there is a natural decomposition
\[
\tcofib(\Phi)
\simeq
\coprod_{x\in \Z^n}\tcofib(\Phi(-;x)),
\]
it suffices to show $\tcofib(\Phi(-;x))\simeq 0$ for every $x\in \Z^n$.

If $x\neq (0,\ldots,0)$,
then there exists $i\in [n]$ such that $x\notin P_i$.
The induced map $\Phi(I;x) \to 
 \Phi(I\cup \{i\};x)$
is a $\Ztwo$-homeomorphism whenever $i\notin I$.
Together with the categorical result \cite[Proposition A.6.5]{logSH},
we have $\tcofib(\Phi(-;x))\simeq 0$.

If $x=(0,\ldots,0)$,
then the induced map $\Phi(I;x)\to \Phi(I\cup \{0\};x)$ is a $\Ztwo$-homeomorphism whenever $0\notin I$.
We have $\tcofib(\Phi(-;x))\simeq 0$ similarly.
\end{proof}

The $\THH$ and $\TC$ parts of the following result are proved in \cite[Theorem 8.4.4]{logSH}.

\begin{thm}
\label{mot.3}
Let $S$ be a noetherian separated scheme.
Then the sheaves
\[
\THH,
\THR^{\Z/2},
\TC,
\TCR^{\Z/2}
\in
\Sh_{sNis}(\SmlSm/S,\Sp)
\]
are in $\logSH_{S^1}(S)$.
\end{thm}
\begin{proof}
Immediate from Theorem \ref{mot.1}.
\end{proof}

Since $\THR$ is not a sheaf of spectra but a sheaf of $\Ztwo$-spectra,
we cannot directly apply the results in \cite{logDM} and \cite{logSH} to $\THR$.
Instead, we will often use the fact that the pair $(i^*,(-)^{\Ztwo})$ is conservative to reduce to the case of $\THH$ and $\THR^{\Z/2}$.

For any vector bundle $\cE\to X$ with $X\in \SmlSm/S$,
the \emph{real topological Hochschild homology of the Thom space $\Th(\cE)$} is defined to be
\begin{equation}
\label{mot.15.1}
\THR(\Th(\cE))
:=
\fib(\THR(\cE)
\to
\THR(\Blow_{Z}\cE,E)),
\end{equation}
where $Z$ is the zero section of $\cE$,
and $E$ is the exceptional divisor.

\begin{prop}
\label{mot.15}
With the above notation,
there are natural equivalences of $\Ztwo$-spectra
\begin{gather*}
\THR(\Th(\cE))
\simeq
\fib(
\THR(\P(\cE\oplus \cO))
\to
\THR(\P(\cE))),
\\
\TCR(\Th(\cE))
\simeq
\fib(
\TCR(\P(\cE\oplus \cO))
\to
\TCR(\P(\cE))).
\end{gather*}
\end{prop}
\begin{proof}
We focus on the case of $\THR$ since the proofs are similar.

Let $Y$ and $Y'$ be the blow-up of $\cE$ and $\P(\cE\oplus \cO)$ along the zero section,
and let $E$ and $E'$ be the exceptional divisors.
By the proof of \cite[Proposition 7.4.5]{logDM},
there exists a commutative diagram
\[
\begin{tikzcd}
(Y,E)\ar[d]\ar[r]&
(Y',E')\ar[d]\ar[r,leftarrow]&
\P(\cE)
\\
\cE\ar[r]&
\P(\cE\oplus \cO)
\end{tikzcd}
\]
such that the induced maps
\begin{gather*}
\fib(\cF(\P(\cE\oplus \cO))\to \cF(Y',E'))
\to
\fib(\cF(\cE)\to \cF(Y,E)),
\\
\cF(Y',E')
\to
\cF(\P(\cE))
\end{gather*}
are equivalences for $\cF\in \logSH_{S^1}(S)$.
In particular,
these are equivalences for $\THH$ and $\THR^{\Z/2}$ by Theorem \ref{mot.3}.
From this,
we obtain the desired equivalence.
\end{proof}

\begin{lem}
\label{mot.16}
Let $\cE_1\to X_1$ and $\cE_2\to X_2$ be vector bundles,
where $X_1,X_2\in \SmlSm/S$.
We regard $T:=\cE_1\times_S \cE_2$ as a vector bundle over $X_1\times_S X_2$.
We set
\[
T_1:=(\Blow_{Z_1} \cE_1,E_1)\times_S \cE_2,
\;
T_2:=\cE_1\times_S (\Blow_{Z_2}\cE_2,E_2),
\;
T_{12}:=T_1\times_S T_2,
\]
where $Z_1$ and $Z_2$ are the zero sections of $\cE_1$ and $\cE_2$.
Then there are natural equivalences of $\Ztwo$-spectra
\begin{gather*}
\THR(\Th(T))
\simeq
\tfib(\THR(C_3)),
\\
\TCR(\Th(T))
\simeq
\tfib(\TCR(C_3)),
\end{gather*}
where $C_3$ is the cartesian square
\[
\begin{tikzcd}
T_{12}\ar[d]\ar[r]&
T_1\ar[d]
\\
T_2\ar[r]&
T.
\end{tikzcd}
\]
\end{lem}
\begin{proof}
Again,
we focus on $\THR$.
Consider the fs log schemes $T_I\in \SmlSm/S$ for every subset $I$ of $\{1,2,4\}$ in \cite[Construction 7.4.14]{logDM},
and consider the induced squares
\[
C_1:=
\begin{tikzcd}
T_4\ar[d]\ar[r]&
T_4\ar[d]
\\
T_4\ar[r]&
T,
\end{tikzcd}
\;\;\;
C_2:=
\begin{tikzcd}
T_{124}\ar[d]\ar[r]&
T_{14}\ar[d]
\\
T_{24}\ar[r]&
T.
\end{tikzcd}
\]
We have the induced maps of squares $C_1\leftarrow C_2\to C_3$.
By Theorem \ref{mot.3} and the proof of \cite[Proposition 7.4.15]{logDM},
the induced maps
\[
\tfib(\cF(C_1))\to \tfib(\cF(C_2))\leftarrow \tfib(\cF(C_3))
\]
are equivalences for $\cF:=\THH,\THR^{\Z/2}$.
This implies that we have a natural equivalence
\[
\tfib(\THR(C_1))\simeq \tfib(\THR(C_3)),
\]
which yields the desired equivalence.
\end{proof}

\begin{prop}
\label{mot.17}
For $X\in \SmlSm/S$ and integer $n\geq 0$,
there are natural equivalences of $\Z/2$-spectra
\begin{gather*}
\THR(\Th(X\times \A^n))
\simeq
\Sigma^{n(\sigma-1)}\THR(X),
\\
\TCR(\Th(X\times \A^n))
\simeq
\Sigma^{n(\sigma-1)}\TCR(X),
\end{gather*}
where we regard $X\times \A^n$ as the rank $n$ trivial bundle over $X$.
\end{prop}
\begin{proof}
Again,
we focus on $\THR$.
We proceed by induction on $n$.
The claim is trivial if $n=0$.
Assume $n>0$.
We set $X_1:=X$,
$\cE_1:=X\times \A^{n-1}$,
$X_2:=S$,
and $\cE_2:=S\times \A^1$.
The square $C_3$ in Lemma \ref{mot.16} becomes
\[
\begin{tikzcd}
(\Blow_X(X\times \A^{n-1}),E_{n-1})\times (\A^1,0)\ar[r]\ar[d]&
(\Blow_X(X\times \A^{n-1}),E_{n-1})\times \A^1\ar[d]
\\
X\times \A^{n-1}\times (\A^1,0)\ar[r]&
X\times \A^{n-1}\times \A^1,
\end{tikzcd}
\]
where $E_{n-1}$ is the exceptional divisor.
By \cite[Proposition 5.2.4]{THR}, Proposition \ref{thrlog.8}, and Lemma \ref{mot.16},
we have a natural equivalence
\[
\THR(\Th(X\times \A^n))
\simeq
\THR(\Th(X\times \A^{n-1}))
\wedge
\fib(\THR(\Sphere[\N])\to \THR(\Sphere[\N],\N)).
\]
Combine this with Example \ref{thrlog.10} and use the induction hypothesis to conclude.
\end{proof}

Next, we are concerned about the Gysin cofiber sequence in logarithmic motivic homotopy theory.
For this,
we recall the deformation to the normal cone construction in the logarithmic setting as follows.
Assume that $X\in \SmlSm/S$ has the form $(\ul{X},Z_1+\cdots+Z_n)$,
and let $Z$ be a strict closed subscheme of $X$ such that $\ul{Z}$ is strict normal crossing with $Z_1+\cdots+Z_n$ in the sense of \cite[Definition 7.2.1]{logDM}.
The \emph{blow-up of $X$ along $Z$} is 
\[
\Blow_Z X:=(\Blow_{\ul{Z}}\ul{X},\widetilde{Z}_1+\cdots+\widetilde{Z}_n),
\]
where $\widetilde{Z}_i$ is the strict transform of $Z_i$ for $1\leq i\leq n$.
The \emph{normal bundle of $Z$ in $X$} is defined to be
\[
\Normal_Z X
:=
\Normal_{\ul{Z}} \ul{X} \times_{\ul{X}} X,
\]
where $\Normal_{\ul{Z}} \ul{X}$ denotes the normal bundle of $\ul{Z}$ in $\ul{X}$.
The \emph{deformation space associated with $Z\to X$} is defined to be
\[
\Deform_Z X
:=
\Blow_{Z\times \{0\}}(X\times \square)
-
\Blow_{Z\times \{0\}}(X\times \{0\}).
\]

\begin{thm}
\label{mot.4}
With the above notation,
there exist natural cofiber sequences of $\Ztwo$-spectra
\begin{gather*}
\THR(\Th(\Normal_Z X))
\to
\THR(X)
\to
\THR(\Blow_Z X,E),
\\
\TCR(\Th(\Normal_Z X))
\to
\TCR(X)
\to
\TCR(\Blow_Z X,E),
\end{gather*}
where $E$ is the exceptional divisor.
\end{thm}
\begin{proof}
Again,
we focus on $\THR$.
We have the induced maps
\begin{align*}
\fib(\THR(X)\to \THR(\Blow_Z X,E))
\leftarrow &
\fib(
\THR(\Deform_Z X)
\to
\THR(\Blow_{Z\times \square}(\Deform_Z X),E^D)
)
\\
\to &
\fib(
\THR(\Normal_Z X)
\to
\THR(\Blow_Z (\Normal_Z X),E^N)
),
\end{align*}
where $E$, $E^D$, and $E^N$ are the exceptional divisors.
It suffices to show that the corresponding maps for $\THH$ and $\THR^{\Z/2}$ are equivalences of spectra since the pair of functors $(i^*,(-)^{\Ztwo})$ is conservative.
This is a consequence of \cite[Theorem 7.5.4]{logDM} (see also \cite[Theorem 3.2.14]{logSH}) and Theorem \ref{mot.3}.
\end{proof}

Recall from \cite[Definition 8.5.3]{logSH} that we have the log motivic spectra
\begin{gather*}
\bTHH:=(\THH,\THH,\cdots),
\\
\bTC:=(\TC,\TC,\cdots)
\end{gather*}
whose bonding maps $\THH\to \Omega_{\P^1}\THH$ and $\TC\to \Omega_{\P^1}\TC$ are obtained by the projective bundle formula.

\begin{exm}
\label{mot.6}
Let $X$ be a scheme,
and let $\cE\to X$ be a rank $n$ vector bundle.
There exists a Thom equivalence
\[
\THH(\Th(\cE))
\simeq
\THH(X).
\]
see \cite[Proposition 8.6.9]{logSH}.
This is a consequence of the fact that $\bTHH$ is an orientable logarithmic motivic spectrum \cite[Definition 7.1.3, Theorem 8.6.7]{logSH}.
Unlike this, in general,
we do not have an equivalence
\[
\THR(\Th(\cE))
\simeq
\Sigma^{n(\sigma-1)}\THR(X),
\]
i.e., $\THR(\Th(\cE))\not \simeq \THR(\Th(X\times \A^n))$ by Proposition \ref{mot.17}.

For example,
if $X:=\P_{\Sphere}^n$ and $Z:=\P_{\Sphere}^{n-1}$ with even $n\geq 2$,
then Proposition \ref{mot.2} and Theorems \ref{mot.1} and \ref{mot.4} yield
\[
\THR(\Th(\Normal_Z X))
\simeq
\bigoplus_{j=1}^{\lfloor n/2 \rfloor}
i_*\Sphere.
\]
This is not equivalent to
\[
\Sigma^{\sigma-1}\THR(Z)
\simeq
\Sigma^{\sigma-1}\bigoplus_{j=1}^{\lfloor (n-1)/2 \rfloor}
i_*\Sphere
\oplus
\Sigma^{n(\sigma-1)}
\]
that is obtained by Proposition \ref{mot.2}.
We have a similar conclusion if $n$ is odd too.
\end{exm}

We also obtain the following descent result with respect to blow-ups along smooth centers (with the trivial involutions).

\begin{thm}
\label{mot.5}
Let $Z\to X$ be a closed immersion of smooth schemes over $S$.
Then the induced squares of $\Ztwo$-spectra
\[
\begin{tikzcd}
\THR(X)\ar[d]\ar[r]&
\THR(Z)\ar[d]
\\
\THR(Z\times_X \Blow_Z X)\ar[r]&
\THR(\Blow_Z X),
\end{tikzcd}
\begin{tikzcd}
\TCR(X)\ar[d]\ar[r]&
\TCR(Z)\ar[d]
\\
\TCR(Z\times_X \Blow_Z X)\ar[r]&
\TCR(\Blow_Z X)
\end{tikzcd}
\]
are cartesian.
\end{thm}
\begin{proof}
Again,
we focus on $\THR$.
The corresponding squares of spectra for $\THH$ and $\THR^{\Z/2}$ are cartesian by \cite[Theorem 7.3.3]{logDM} (see also \cite[Theorem 3.3.5]{logSH}) and Theorem \ref{mot.3}.
This implies the claim since the pair of functors $(i^*,(-)^{\Ztwo})$ is conservative.
\end{proof}

Let $\P^\sigma$ be the scheme $\P^1$ with the involution given by $[x:y]\mapsto [y:x]$.
Let $1$ be the base points of $\P^1$ and $\P^\sigma$,
and then we form the endofunctors
$\Omega_{\P^1}$, $\Omega_{\P^\sigma}$, and $\Omega_{\P^1\wedge \P^\sigma}\simeq \Omega_{\P^\sigma}\Omega_{\P^1}$ on $\PSh((\SmlSm/S)_{\Ztwo},\Sp)$ and $\PSh((\SmlSm/S)_{\Ztwo},\Sp^{\Ztwo})$.

\begin{prop}
\label{mot.7}
Let $X$ be a noetherian separated fs log scheme.
Then there are natural equivalences of $\Ztwo$-spectra
\begin{gather*}
\THR(X)
\simeq
\Omega_{\P^1\wedge \P^\sigma}\THR(X),
\\
\TCR(X)
\simeq
\Omega_{\P^1\wedge \P^\sigma}\TCR(X).
\end{gather*}
\end{prop}
\begin{proof}
Again,
we focus on $\THR$.
Following Remark \ref{mot.12},
we reduce to \cite[Proposition 5.1.5]{THR}.
\end{proof}

\begin{df}
\label{mot.8}
The \emph{$\infty$-category of prelogarithmic motivic $\Ztwo$-spectra over $S$} is defined to be
\[
\prelogSH^{\Ztwo}(S)
:=
\Sp_{\P^1\wedge \P^\sigma}
((\P^\bullet,\P^{\bullet-1})^{-1}\Sh_{seNis}((\SmlSm/S)_{\Ztwo},\Sp)).
\]
\end{df}

\begin{rmk}
\label{mot.9}
Let $\A^w$ denote the affine line $\A^1$ with the involution $w$ given by $x\mapsto -x$.
We have the $\A^1$-homotopy
\begin{equation}
\label{mot.9.1}
\A^1\times \A^w
\to
\A^w
\end{equation}
given by $(x,y)\mapsto xy$.
Due to this,
the projections
$X\times \A^w \to X$ are automatically inverted if we invert the projections $X\times \A^1\to X$ for all $X\in (\Sm/S)_{\Ztwo}$.

However,
this phenomenon is not generalized to the logarithmic setting since there exists no morphism of fs log schemes
\[
\square \times \square^w \to \square^w
\]
extending \eqref{mot.9.1},
where $\square^w$ denotes the fs log scheme $\square$ with the involution $[x,y]\mapsto [-x,y]$.
This indicates that inverting the projections $X\times \square^w\to X$ for all $X\in (\SmlSm/S)_{\Ztwo}$ is at least required for a better behaved $\infty$-category of logarithmic motivic $\Ztwo$-spectra.
\end{rmk}

\begin{df}
\label{mot.10}
The \emph{prelogarithmic motivic $\Ztwo$-fixed point functor} is defined to be the functor
\[
(-)^{\Ztwo}
\colon
\prelogSH^{\Ztwo}(S)
\to
\logSH(S)
\]
sending a $\P^1\wedge \P^\sigma$-spectrum $(\cF_0,\cF_1,\cF_2,\ldots)$ to the $\P^1$-spectrum
\[
(\cF_0,
\Omega_{\P^\sigma}\cF_1,
\Omega_{\P^{2\sigma}} \cF_2,\ldots
)
\]
restricting to the objects of $\SmlSm/S$ with the trivial involutions,
where the bonding map $\Omega_{\P^{n\sigma}}\cF_n\to \Omega_{\P^1\wedge\P^{(n+1)\sigma}}\cF_{n+1}$ is induced by the bounding map $\cF_n\to \Omega_{\P^1\wedge \P^\sigma}\cF_{n+1}$.
\end{df}

\begin{df}
\label{mot.11}
The \emph{motivic real topological Hochschild $\Ztwo$-spectrum} and \emph{motivic real topological cyclic $\Z/2$-spectrum} are the $\P^1\wedge \P^\sigma$-spectra
\begin{gather*}
\bTHR
:=
(\THR^{\Z/2},\THR^{\Z/2},\cdots)
\in
\prelogSH^{\Ztwo}(S),
\\
\bTCR
:=
(\TCR^{\Z/2},\TCR^{\Z/2},\cdots)
\in
\prelogSH^{\Ztwo}(S),
\end{gather*}
whose bonding maps are given by Proposition \ref{mot.7}.
We have the induced motivic spectra
\[
\bTHR^{\Ztwo},
\bTCR^{\Ztwo}
\in
\logSH(S).
\]
\end{df}

\begin{df}
\label{mot.13}
For a strict Nisnevich sheaf of spectra $\cF$ on $(\SmlSm/S)_{\Ztwo}$,
let $i^*\cF$ be the strict Nisnevich sheaf of spectra on $\SmlSm/S$ given by
\[
i^*\cF(X)
:=
\cF(i_\sharp X).
\]
Observe that for every integer $n\geq 1$,
$i^*\cF$ is $(\P^n,\P^{n-1})$-invariant if $\cF$ is $(\P^n,\P^{n-1})$-invariant.
We have the ``forgetful'' functor
\[
i^*
\colon
\prelogSH^{\Ztwo}(S)
\to
\logSH(S)
\]
sending a $\P^1\wedge \P^\sigma$-spectrum $(\cF_0,\cF_1,\cF_2,\ldots)$ to the $\P^1$-spectrum
\[
(i^*\cF_0,\Omega_{\P^\sigma}i^*\cF_1,\Omega_{\P^{2\sigma}} i^*\cF_2,\ldots)
\]
with the induced bonding maps.
\end{df}

\begin{prop}
\label{mot.14}
There are equivalences in $\logSH(S)$
\begin{gather*}
i^*\bTHR
\simeq
\bTHH,
\\
i^*\bTCR
\simeq
\bTC.
\end{gather*}
\end{prop}
\begin{proof}
This is a consequence of Corollary \ref{descent.13} and Proposition \ref{TCR.24}.
\end{proof}

\begin{rmk}
In this section,
we have considered mainly $\THR$ and $\TCR$,
but we can similarly show the analogous results for $\TCR^-:=\THR^{hS^\sigma}$, $\TCR_{hS^\sigma}$, and $\TPR:=\THR^{tS^\sigma}$.
\end{rmk}

\bibliography{logTHR}
\bibliographystyle{siam}
\end{document}